\documentclass[11pt]{article}

\usepackage{latexsym,mathrsfs}
\usepackage{amsmath,amssymb}
\usepackage{amsthm,enumerate,verbatim}
\usepackage{amsfonts}
\usepackage{graphicx}
\usepackage{algorithm}
\usepackage{algorithmic}
\usepackage{url}
\usepackage{here}
\usepackage{geometry}
\usepackage{rotating}

\renewcommand{\algorithmicrequire}{\textbf{Input:}}
\renewcommand{\algorithmicensure}{\textbf{Output:}}

\newtheorem{conj}{Conjecture}

\newtheorem*{cor*}{Corollary}

\newtheorem*{definition*}{Definition}

\newtheorem{remark}{Remark}

\DeclareMathOperator{\rank}{rank}

\title{Heuristics for Exact Nonnegative Matrix Factorization}

\date{} 

\author{Arnaud Vandaele\thanks{Department of Mathematics and Operational Research, 
Facult\'e Polytechnique, Universit\'e de Mons, Rue de Houdain~9, 7000 Mons, Belgium. 
Emails: 
 \texttt{\{arnaud.vandaele, nicolas.gillis, daniel.tuyttens\}@umons.ac.be.}} \and Nicolas Gillis$^*$  
\and Fran\c{c}ois Glineur\thanks{Universit\'e catholique de Louvain, CORE and ICTEAM Institute, 
B-1348 Louvain-la-Neuve, Belgium; \texttt{francois.glineur@uclouvain.be}.} \and  Daniel Tuyttens$^*$  
}

\begin{document}

\maketitle

\begin{abstract} 
The exact nonnegative matrix factorization (exact NMF) problem is the following: 
given an $m$-by-$n$ nonnegative matrix $X$ and a factorization rank $r$, find, if possible, an $m$-by-$r$ nonnegative matrix $W$ and an $r$-by-$n$ nonnegative matrix $H$ such that $X = WH$. 
In this paper, we propose two heuristics for exact NMF, one inspired from simulated annealing and the other from the greedy randomized adaptive search procedure. 
We show that these two heuristics are able to compute exact nonnegative factorizations for several classes of nonnegative matrices (namely, linear Euclidean distance matrices, slack matrices, unique-disjointness matrices, and randomly generated matrices) and as such demonstrate their superiority over standard multi-start strategies. 
We also consider a hybridization between these two heuristics that allows us to combine the advantages of both methods. 
Finally, we discuss the use of these heuristics to gain insight on the behavior of the nonnegative rank, i.e., the minimum factorization rank such that an exact NMF exists. In particular, we disprove a conjecture on the nonnegative rank of a Kronecker product, propose a new upper bound on the extension complexity of generic $n$-gons and conjecture the exact value of (i) the extension complexity of regular $n$-gons and (ii) the nonnegative rank of a submatrix of the slack matrix of the correlation polytope. 
\end{abstract} 


\section{Introduction} \label{intro}


Nonnegative matrix factorization (NMF) is the problem of finding good approximations of a given nonnegative matrix as a  low-rank product of two nonnegative matrices. This linear dimensionality reduction technique has been used very successfully for a large variety of machine learning and data mining tasks, including text mining and image processing \cite{LS99}. Formally, given a nonnegative matrix $X \in \mathbb{R}^{m \times n}_+$ and a factorization rank $r$, NMF looks for two nonnegative matrices 
$W \in \mathbb{R}^{m \times r}_+$ and 
$H \in \mathbb{R}^{r \times n}_+$ such that $X \approx WH$. Despite the fact that NMF is NP-hard in general~\cite{V09}, it has been used successfully in many practical situations. A large number of dedicated nonlinear local optimization schemes have been developed to compute good factorizations \cite{cichocki2009nonnegative}, e.g., to try identifying good local minima of the following nonconvex optimization problem 
\[
\min_{W \in \mathbb{R}^{m \times r}, H \in \mathbb{R}^{r \times n}} ||X - WH||_F^2 \quad \text{ such that } W \geq 0 \text{ and } H \geq 0 . 
\]
Nearly all NMF algorithms are iterative: at each step, they aim to improve the current solution. In practice, these algorithms are usually initialized randomly, or with some ad hoc strategies; see, e.g., \cite{CAZP09, G14} and the references therein.  

Comparatively, much less attention has been given in the literature to the development of heuristic algorithms aimed at finding better local minima of the NMF approximation problem. 
In this paper, we tackle the problem of computing high quality local minima for the NMF problem. In particular, our focus is on finding exact nonnegative factorizations, that is, computing nonnegative factors $W$ and $H$ such that $X = WH$ holds exactly, a problem we will refer to as \emph{exact NMF}. 
The minimum factorization rank for which such an exact NMF exists is called the \emph{nonnegative rank} of $X$ and is denoted $\rank_+(X)$ \cite{CR93}. 

\subsection{Motivating Applications}

For machine learning and data mining applications, it typically does not make sense to look for exact NMF's because the data is usually contaminated with noise. 
However, several other applications are closely related to exact NMF and the nonnegative rank, including the following. 
\begin{itemize}

\item \emph{Computing the minimum biclique cover number of a bipartite graph}. 
Let $G = (V = V_1 \cup V_2 ,E)$ be a bipartite graph with $V_1 = \{s_1,s_2,\dots,s_m\}$, $V_2 = \{t_1,t_2,\dots,t_n\}$ and $E \subseteq V_1 \times V_2$. 
A complete bipartite subgraph of $G$, referred to as a biclique, 
is a subgraph 
$G' = ( V'_1 \cup V'_2 , E')$ with $V'_1 \subseteq V_1$, $V'_2 \subseteq V_2$ and $E' \subseteq E$ such that $E' = V'_1 \times V'_2$, 
that is, all vertices in $V'_1$ and $V'_2$ are connected. 
The minimum biclique cover number bc$(G)$ of $G$ is the minimum number of bicliques needed to cover all edges in $G$. 
Let $X_G \in \{0,1\}^{m \times n}$ be the biadjacency matrix of the graph, that is, $X_G(i,j) = 1$ for all pairs $(i,j)$ such that $(s_i, t_j) \in E$. A biclique of $G$ corresponds to a nonzero combinatorial rectangle in the biadjacency matrix $X_G$. 
Given an exact NMF of the biadjacency matrix $X_G = WH = \sum_k W(:,k)H(k,:)$, the nonzero pattern of each rank-one factor $W(:,k)H(k,:)$ must correspond to a biclique. In fact, because $W$ and $H$ are nonnegative, $X_G(i,j) = 0 \Rightarrow W(i,k)H(k,j) = 0$ for all $k$. 
Moreover, the union of the bicliques corresponding to the rank-one factors must cover $G$ since $X_G = WH$. 
Therefore, 
\[ 
\text{bc}(G) \leq \rank_+(X_G), 
\]
and computing an exact NMF of $X_G$ provides an upper bound for the minimum biclique cover of $G$. 
Conversely, a minimum biclique cover of $G$ provides a lower bound for the nonnegative rank of $X_G$, which is referred to as the \textit{rectangle covering bound} and is denoted $\text{rc}(X_G) = \text{bc}(G)  \leq \rank_+(X_G)$; see \cite{FK11} and the references therein.

\item \emph{Computing the extension complexity of polyhedra}. Given a polytope $\mathcal{P}$, an extension  (or lift, or extended formulation) of $\mathcal{P}$ is a higher-dimensional polytope $\mathcal{Q}$ for which there exists a linear projection $\pi$ such that $\pi(\mathcal{Q}) = \mathcal{P}$. If the number of facets of $\mathcal{P}$ is large (possibly growing exponentially with the dimension $k$), a crucial question in combinatorial optimization is whether there exists an extension with a small number of facets (ideally bounded by a polynomial in the dimension $k$), in which case a linear program 
over $\mathcal{P}$ can solved much more effectively using an equivalent formulation over $\mathcal{Q}$.  
The minimum number of facets appearing in any extension of $\mathcal{P}$ is called the extension complexity of $\mathcal{P}$. 
In \cite{Y91}, Yannakakis proved that the extension complexity of a polytope $\mathcal{P}$ is equal to the nonnegative rank of its slack matrix $S_{\mathcal{P}}$ (see Section 2 for a definition of the slack matrix of a polytope). It is also worth mentioning that any exact NMF of $S_{\mathcal{P}}$ provides an explicit extension for $\mathcal{P}$. 
This result has been extensively used recently to prove bounds on the extension complexity of well-known polytopes; 
see \cite{CCCZ09, K11, GPT13} and the references therein.  
For example, it was shown very recently that the perfect matching polytope admits no polynomial-size extension \cite{Ro14}, answering a  long-standing open question in combinatorial optimization.  

\item \emph{Conjecturing new theoretical results on the nonnegative rank, or disproving them}. 
As any exact NMF of a nonnegative matrix provides an upper bound on its nonnegative rank, one can use this technique to find counter-examples to conjectures dealing with the nonnegative rank of matrices, or strengthen our belief that some conjectures are correct. For example, Beasly and Laffey \cite{BL09} developed some lower bounding techniques for the nonnegative rank of $n$-by-$n$ linear Euclidean distance matrices 
(see Section~\ref{msnmfalgo} for more details) and conjectured that these rank-three matrices have nonnegative rank $n$. Using a standard NMF algorithm combined with a simple multi-start heuristic, Gillis and Glineur~\cite{GG10b} found a counterexample: 
the 6-by-6 linear Euclidean distance matrix $M(i,j) = (i-j)^2$ ($1 \leq i,j \leq 6$) has nonnegative rank five, which disproved the conjecture and motivated the development of stronger lower bounds for the nonnegative rank of such matrices.  Along the same line, Hrube{\v{s}}~\cite{H12} developed some new upper bounding techniques for the nonnegative rank of such matrices. 
In Section~\ref{conject}, we discuss several examples where the use of exact NMF algorithms allows us to  
gain insight on the nonnegative rank.  
\end{itemize} 
Other problems closely related to nonnegative rank and exact NMF computations arise in 
communication complexity \cite{Lee09},  
probability \cite{CR10} and 
computational geometry \cite{GG10b}; see also, e.g., \cite{G14} and the references therein.

\subsection{Computational Complexity} \label{ccomp}


Given an $m$-by-$n$ nonnegative matrix $X$, Vavasis~\cite{V09} proved that checking whether $\rank(X) = \rank_+(X)$ is NP-hard. Therefore, unless $P=NP$, no algorithm  can decide whether $\rank(X) = \rank_+(X)$ using a number of arithmetic operations bounded by a polynomial in $m$, $n$ and $\rank(X)$. 
Nevertheless, Arora et al.~\cite{AGKM11} showed that checking whether a nonnegative matrix admits an exact rank-$r$ NMF can be done in time polynomial in $m$ and $n$ (i.e., considering the factorization rank $r$ fixed). This result relies on a clever reformulation of the exact NMF problem for an $m \times n$ matrix as a system of $\mathcal{O}(mn)$ fixed-degree polynomial equalities involving $\mathcal{O}(r^2 2^r)$ variables. This, combined with the fact that a system of $k$ polynomial inequalities up to degree $d$ and in $p$ variables can be solved in $\mathcal{O}( (kd)^p )$ operations, shows that checking the existence of an exact rank-$r$ NMF can be done with total complexity $\mathcal{O}( (mn)^{r^2 2^r})$. 

This complexity was later improved by Moitra \cite{Moit13} to $\mathcal{O}( (mn)^{r^2} )$. Unfortunately, because they rely on quantifier elimination, these results do not translate in practical algorithms, even when dealing with very small matrices. For example, 
we were unable to compute a  rank-three NMF of a 4-by-4 matrix (which is actually the first non-trivial case since $\rank(X) = 2 \Leftrightarrow \rank_+(X) = 2$ \cite{Tho, CR93}) using either the built-in polynomial equation solver of Mathematica (which runs out of memory after performing a large number of operations) or the qepcad software \cite{brown2003qepcad} dedicated to quantifier elimination. 

It is therefore not clear whether these theory-oriented complexity results can prove useful to perform  exact NMF, even for small-scale matrices, which prompted us to introduce the use of heuristics to tackle  the problem.

\subsection{Contribution and Outline of the Paper} 


The paper is organized as follows. 
Section~\ref{exactNMFds} lists the classes of nonnegative matrices 
on which we benchmark exact NMF algorithms, and provides a description of the corresponding applications.   
Section~\ref{msnmfalgo} presents two multi-start strategies, and compares their combination with several initializations strategies and state-of-the-art NMF algorithms. This allows us to select an NMF algorithm (that is, a method to locally improve a current solution) and initialization strategies for the rest of the paper. 
 Section~\ref{sagrasp} introduces two heuristics dedicated to exact NMF (SA and RBR) along with a hybridization. 
Section~\ref{ne} compares these heuristics, showing that they outperform multi-start strategies. In particular, RBR performs remarkably well and is able to identify exact NMF's very efficiently for several classes of matrices, 
while SA and the hybridization strategy are able to compute an exact NMF for all considered matrices. 
Section~\ref{ne} also discusses the limitations of these approaches, which are unable to compute exact NMF's for large and difficult matrices (as expected by the computational complexity of the problem). 
Finally, Section~\ref{conject} discusses the use of these heuristics to better understand the nonnegative rank. In particular, we propose new conjectures for the nonnegative rank of 
(i) the Kronecker product of two nonnegative matrices,  
(ii) the slack matrices of regular and generic $n$-gons, and 
(iii) a submatrix of the slack matrix of the correlation polytope.  \\

To summarize, the main contributions of the paper are threefold: 
\begin{itemize}

\item Design of two heuristics for exact NMF, along with a hybridization strategy, that outperform multi-start strategies. 
To the best of our knowledge, the only heuristic algorithms previously designed for NMF were developed in \cite{JT11a, JT11b, JT11c} and focused only on the (approximate) NMF problem, and not its exact counterpart. 

\item Comparison of these heuristics with two simple multi-start strategies on several classes of nonnegative matrices for which exact factorizations are relevant for applications. This is to the best of our knowledge the first time  exact NMF algorithms are benchmarked on this type of nonnegative matrices (previous work focused on randomly generated matrices, or on machine learning data sets for which exactness of the factorization is not relevant). 

\item Several examples of the concrete use of these heuristics to address open theoretical questions related to the nonnegative rank.
\end{itemize}

The code and data sets used in the paper have been made available online at 
\begin{center}
\url{https://sites.google.com/site/exactnmf} 
\end{center}
We hope that the promising results showed by the methods introduced in this paper will motivate researchers to further develop even faster and more effective heuristics for exact NMF. 


\section{Benchmark Nonnegative Matrices for Exact NMF} \label{exactNMFds}
 
Throughout the paper, we will compare exact NMF algorithms on the following nonnegative matrices (see Table~\ref{datasets}): 
\begin{itemize}
\item \emph{Linear Euclidean Distance Matrices}. 
Given a set of real numbers $a_i$ for $1 \leq i \leq n$, a linear Euclidean distance matrices (EDM) is defined as 
\[
X_a(i,j) = (a_i-a_j)^2 \quad \text{ for } 1 \leq i \leq n \text{ and } 1 \leq j \leq n. 
\]
If at least three entries of $a \in \mathbb{R}^n$ are distinct, then $\rank(M) = 3$. 
However, the nonnegative rank of linear EDM's can be arbitrarily large: in fact, it was proved in \cite{BL09} that, if the entries of $a$ are distinct, 
\begin{align*}
\rank_+(X_a)  
& \quad \geq \quad \min \left\{ k \; \Big| \; \binom{k}{\lfloor k/2 \rfloor} \geq n  \right\}  \quad \geq \quad \log_2(n). 
\end{align*}
The lower bound was later improved in \cite{GG10b}. 
A subclass of linear EDM's are the following 
\[
X_{[n]}(i,j) = (i-j)^2 \quad \text{ for } 1 \leq i \leq n \text{ and } 1 \leq j \leq n,  
\]
for which it was proved in \cite[Th.1]{H12} that 
\begin{equation*}
\rank_+(X_{[2n]}) \leq \rank_+(X_{[n]}) + 2. 
\end{equation*}
If $n$ is a power of two, we therefore have $\rank_+(X_{[n]}) \leq 2 \log_2(n)$ since $\rank_+(X_{[2]}) = 2$. 
Combining this upper bound with the lower bound from \cite{GG10b} allows us to determine the nonnegative rank for these matrices up to $n = 16$; see Table~\ref{datasets}. However, as we will see later on, it is non-trivial to compute exact NMF for these matrices.

\item \emph{Slack Matrices}. 
The slack matrix of a polytope $\mathcal{P}$ with $m$ facets and $n$ vertices is defined as the $m \times n$ nonnegative matrix $S_\mathcal{P}$ whose $(i,j)$th entry $S_\mathcal{P}(i,j)$ is equal to the slack of the $j$th vertex with respect to the $i$th facet. Formally, given the list of $n$ vertices $v_j$ ($1 \le j \le n$) and a facet description of the polytope $\mathcal{P} = \{ x \in \mathbb{R}^k \ | \ A(i,:)x \leq b_i \text{ for } 1 \leq i \leq m \}$, we have that 

\[
S_{\mathcal{P}}(i,j) = A(i,:) v_j - b(i) \geq 0 \text{ for } 1 \leq i \leq m \text{ and } 1 \leq j \leq n. 
\]
As recalled in the Introduction, the nonnegative rank of the slack matrix of $\mathcal{P}$ is equal to the extension complexity of $\mathcal{P}$.  In this paper, we use slack matrices of several  well-known classes of polytopes; see Table~\ref{datasets}.

\item \emph{Unique-Disjointness Matrices}. A unique-disjointness (UDISJ) matrix $X_n \in \{0,1\}^{2^n \times 2^n}$ of order $n$ is a matrix whose rows and columns are indexed by all vectors in $a, b \in \{0,1\}^n$ and which satisfies 
\[
X_n(a,b) 
= \left\{ 
\begin{array}{cc}
1 & \text{ if } a^T b = 0, \\
0 & \text{ if } a^T b = 1,\\ 
? & \text{ otherwise},\\ 
\end{array}
\right. 
\]
where $?$ means that the corresponding entry can take any (nonnegative) value.  UDISJ matrices have been successfully used to prove lower bounds on the extension complexity of polytopes, because their sparsity pattern can be found in submatrices of several interesting slack matrices; see, e.g., \cite{KW13} and the references therein. 
Note that UDISJ matrices also often appear in the communication complexity literature\footnote{Bob is given $a$, Alice $b$, 
and they have to decide whether $a^Tb \neq 0$ while minimizing the number of bits exchanged; see \cite{Lee09} for more details.}. 
Many of the best lower bounds for UDISJ matrices are based on the rectangle covering bound (see Section~\ref{intro}), and the class of matrices we consider here is built on a similar principle (we choose not to use the UDISJ matrices themselves as their nonnegative rank is not known exactly).  

Given $r$ rank-one binary (combinatorial) rectangles $w_k h_k^T \in \{0,1\}^{2^n \times 2^n}$ ($1 \leq k \leq r$) covering $X_n$, we define 
\[
Y_n = \sum_{k=1}^r w_k h_k^T \; \in \;  \{0,1,\dots,r\}^{2^n \times 2^n}. 
\] 
Matrix $Y_n$ features the same sparsity pattern as $X_n$ (that is, $Y_n(i,j) \neq 0 \Leftrightarrow X_n(i,j) \neq 0$ for all $i,j$). We can verify that $\rank(Y_n)=r$ and since these matrices clearly admit a rank-$r$ NMF, we can conclude that $\rank_+(Y_n) = r$, and we will use those matrices $Y_n$ for our benchmark; see Table~\ref{datasets}.

\begin{table}[ht] 
\begin{center}
\caption{Nonnegative matrices used to compare exact NMF heuristics.} 
\begin{tabular}{|c||c|c|c|c|c|}
\hline
					& $m$  &   $n$ & $\rank(X)$   & $\rank_+(X)$ & Abreviation  \\ 
\hline  
                   &	 6   & 6  & 3& 5 & LEDM6\\ 
			 Linear EDM's							 &	 8   & 8  &3 &6 & LEDM8  \\ 
						 $X(i,j) = (i-j)^2$, 				&	   12 & 12  & 3 &  7 & LEDM12\\ 
						for $1 \leq i \leq m$, 	$1 \leq j \leq n$			&	  16  & 16  &3 &8  & LEDM16 \\ 
										&	 32   & 32  &3 & 10$^*$  & LEDM32 \\ 
\hline  
Slack Matrix of the Hexagon      &	6  & 6  & 3 & 5  & 6-G \\ 
Slack Matrix of the Heptagon     &	7  & 7  & 3 & 6  & 7-G\\ 
Slack Matrix of the Octagon      &	8  & 8  & 3 & 6  & 8-G\\ 
Slack Matrix of the Nonagon      &	9  & 9  & 3 & 7  & 9-G\\ 
Slack Matrix of the Hexadecagon  &	16 & 16 & 3 & 8  & 12-G\\ 
Slack Matrix of the 32-gon        &	32 & 32 & 3 & 10 & 32-G\\ 
Slack Matrix of the dodecahedron &	20 & 12 & 4 & 9  & 20-D\\ 
Slack Matrix of the 24-cell      &	24 & 24 & 5 & 12$^*$ & 24-C \\ 
\hline  
UDISJ ($n = 4$) &	 16   & 16  & 9  & 9   & UDISJ4 \\ 
UDISJ ($n = 5$) &	 32   & 32  & 18  & 18 & UDISJ5 \\ 
UDISJ ($n = 6$) &	 64   & 64  & 27 &  27 & UDISJ6 \\ 
\hline 
 Randomly generated Matrices: $X = WH$ &	    &   &  &   & \\     
\texttt{density = 0.1}  &	  50  &  50 & 10 & 10 & RND1 \\    
 \texttt{density = 0.3}  &	50    &  50 & 10 & 10 & RND3 \\  \hline  
\end{tabular}
\label{datasets} 
\end{center} 
\vspace{-0.2cm} 
\footnotesize The symbol $^*$ means that the exact value of the nonnegative rank is still unknown, i.e.,  the best known lower bound does not match the best known upper bound. (For LEDM 32, the best lower bound is 9 while for 24-C it is 10.) 
However, after running our heuristics extensively 
on these matrices, we 
believe that all values of the nonnegative ranks appearing in this table are correct. 
\end{table}

\item \emph{Randomly Generated Matrices}.  It is standard in the NMF literature to use randomly generated matrices to compare  algorithms (see, e.g., \cite{KP11}), with the nice feature that the resulting nonnegative rank of these matrices can be specified. For example, generating each entry of $W \in \mathbb{R}^{m \times r}$ and $H \in \mathbb{R}^{r \times n}$ 
uniformly at random in the interval [0,1] and computing $X = WH$ generates, with probability one, a nonnegative matrix $X$ such that $\rank(X) = \rank_+(X) = r$. In this paper, we have generated such matrices of dimensions 50-by-50 with nonnegative rank 10. 
More precisely,  matrix $W$ is generated as follows: 
\begin{itemize}
\item[(i)] generate $W$ such that each column of the 50-by-10 matrix $W$ has exactly one non-zero entry whose location is randomly chosen and its value is picked uniformly at random in the interval [0,1] (this ensures each rank-one factor to be non-zero), and 
\item[(ii)]  add a sparse uniformly distributed (in the interval [0,1]) random update to $W$, with prespecified density $d$ (i.e. apply \texttt{W = W + sprand(50,10,d)})
\end{itemize}
We use $d = 0.1$ and $0.3$ as specified in Table~\ref{datasets}. Matrix $H$ is generated in the same way.  
It turns out that these nonnegative products $W H$ are relatively easy to factorize: in fact, most initializations lead most NMF algorithms to an exact NMF.  Hence these matrices are not very useful to compare exact NMF heuristics ; nevertheless we include them in our comparisons to illustrate this fact.  
\end{itemize}

\section{Designing Heuristics: Key Ingredients and Multi-Start Examples} \label{msnmfalgo} 

Before presenting our proposed heuristics, we explore two multi-start strategies (Section~\ref{msstrategies}). This allows us to discuss some key aspects for comparing and designing such heuristics. 
There are four main building blocks for our proposed heuristics:  
\begin{itemize}
\item the initialization strategy (Section~\ref{rdninit}), 

\item the main algorithm, i.e. the heuristic constructing exact NMF's after applying the initialization strategy, which relies on a local NMF algorithm (Section~\ref{msstrategies} for the multi-start strategies, and Section~\ref{sagrasp} for our two proposed heuristics),  

\item the NMF algorithm used to improve solutions locally (Section~\ref{selectNMFalgo}), 
 
\item a final refinement step that will try to further improve the output of the main algorithm as far as possible (ideally, until an exact NMF is found); see the description Algorithm~\ref{alg:iterations}, which also relies on the local NMF algorithm.
\end{itemize}

This final refinement procedure will be applied to all solutions generated by the heuristics. In this paper, we use a tolerance for the relative error equal to $10^{-6}$, that is, we will assume that an exact NMF $(W,H)$ of $X$ is found as soon as $\frac{\|X-WH\|_F}{\|X\|_F} \leq 10^{-6}$. 
Algorithm~\ref{alg:iterations} runs a local NMF algorithm as long as the relative error decreases at least by a predefined factor $\alpha$ after every period of  $\Delta t$ seconds, otherwise it stops and returns the current solution. We set the parameters to the following rather conservative values: $\Delta t = 1$ second (which is quite large for small matrices\footnote{For example, for a 50-by-50 matrix and $r= 10$, running standard multiplicative updates for one second allows to perform about 10000 iterations on a standard laptop.}) and $\alpha = 0.99$; see Appendix~\ref{appA} for some additional numerical results comparing different values for $\Delta t$ and $\alpha$.

\renewcommand{\thealgorithm}{FR} 
\begin{algorithm}[ht!] 
	\caption{Final Refinement$(X,W,H,\alpha,\Delta t)$}
	\label{alg:iterations}
	\begin{algorithmic}[1]
		\REQUIRE $X \in {\mathbb R}^{m\times n}_+$, $W \in {\mathbb R}^{m\times r}_+$, $H \in {\mathbb R}^{r\times n}_+$, $0<\alpha<1$, $\Delta t$. 
		\medskip 
		
		\STATE $i=1$, $e_0 = +\infty$, $e_1 = {\|X-WH\|_F}/{\|X\|_F}$. 
		\WHILE{$e_{i} < \alpha e_{i-1}$ and $e_{i} > 10^{-6}$}
		  \STATE $i \gets i + 1$. 
			\STATE $[W,H] \gets$ AlgoNMF$(X,W,H,\Delta t)$. \emph{\% See Section~\ref{selectNMFalgo}}
			\STATE $e_i \gets {\|X-WH\|_F}/{\|X\|_F}$.   
		\ENDWHILE
		\RETURN $W \in {\mathbb R}^{m\times r}_+$, $H \in {\mathbb R}^{r\times n}_+$, relative error $e_i$. 
	\end{algorithmic}
\end{algorithm}

Figure~\ref{heurbox} shows how these blocks are arranged in our design of heuristics. 
\begin{figure}[ht!] 
\centering 
   \includegraphics[width=10cm]{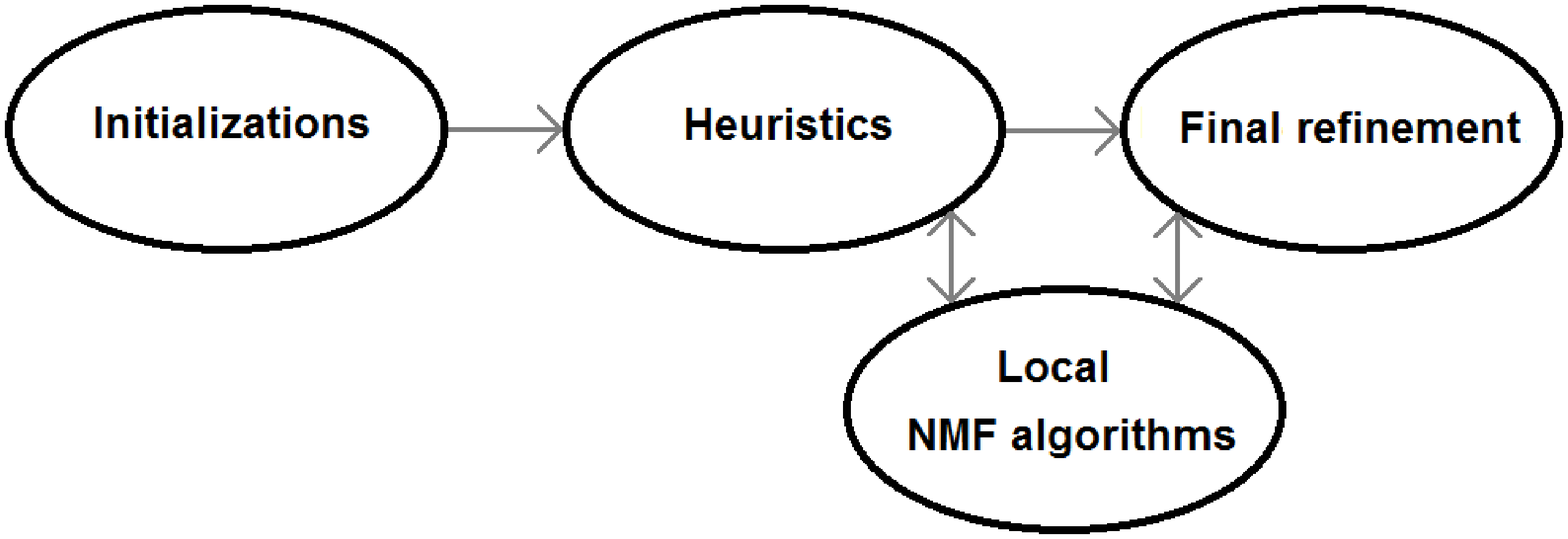} 
\caption{Representation of the stream of the exchange of information between the different building blocks of our exact NMF heuristics. An arrow represent the transfer of a solution.} 
\label{heurbox}
\end{figure} 
Since it will not be practical to display results for all possible combinations of heuristics (there will be five in total), NMF algorithms (five) and initializations (five), along with the different tuning parameters, 
another goal of this section is to select, for the rest of the paper, reasonable values for the parameters of a good multi-start heuristic, along with an efficient local improvement algorithm and initialization strategy that performs well on most examples.

\begin{remark}[Are our exact NMF's really exact?]  \label{trullyexact}
At this point, it is important to insist on the fact that all the numerical experiments performed in this paper are with floating point arithmetic. Hence, we consider a factorization exact if the relative error ${\|X-WH\|_F}/{\|X\|_F}$ is smaller than some threshold (we choose $10^{-6}$) so that the computed factorizations are not exact but high precision solutions. 
It is interesting to point out that all the solutions that we computed with relative error smaller than $10^{-6}$ could be further improved with additional iterations of Algorithm~\ref{alg:iterations} to $10^{-16}$  (which is the smallest possible using the standard Matlab precision). 
Note that it is an open question whether the nonnegative rank over the rationals equals the nonnegative rank over the reals; see, e.g., the discussion in~\cite{V09}. Note that they were recently shown to be different for the semidefinite rank, a generalization of the nonnegative rank~\cite{GFR14}; see \cite{FGP14} for more details.
\end{remark}

\subsection{Two Multi-Start Heuristics} \label{msstrategies}

In this section we propose two multi-start heuristics (Sections~\ref{secms1} and~\ref{secms2}). 

\subsubsection{Multi-Start 1} \label{secms1}

The simplest multi-start strategy one can think of is to restart Algorithm~\ref{alg:iterations} with many different initial matrices until an exact NMF is obtained; see Algorithm~\ref{alg:ms1}. 
Note that this heuristic is the one used in \cite{GG10b} to compute exact NMF of linear EDM's. 
Note also that, in view of Figure~\ref{heurbox}, MS1 corresponds to a heuristic which is an `empty box' that transfers directly the solution from `Initializations' to `Final refinement'. 
\renewcommand{\thealgorithm}{MS1}  
\begin{algorithm}[ht!] 
	\caption{Multi-Start~1$(X,r,\alpha,\Delta t)$}
	\label{alg:ms1}
	\begin{algorithmic}[1]
		\REQUIRE $X \in {\mathbb R}^{m\times n}_+$, $r<\min(m,n)$, $0<\alpha<1$, $\Delta t$, $\text{tol} = 10 ^{-6}$. 
		\medskip 

		\STATE $(W_0,H_0) \gets$ random initialization$(m,n,r)$. \emph{\% See Section~\ref{rdninit}}
		\STATE $[W,H,e] \gets$ Final Refinement$(X,W_0,H_0,\alpha,\Delta t)$. 
	\end{algorithmic}
\end{algorithm}

\subsubsection{Multi-Start 2}  \label{secms2}

Applying Algorithm~\ref{alg:iterations} until convergence is useless when the error does not converge to zero, 
that is, when $(W,H)$ converges to a local minimum with error strictly larger than zero.  
The idea behind Algorithm~\ref{alg:ms2} is to keep the pairs $(W,H)$ with the best potential to obtaining an exact NMF, and therefore avoiding waste of computational time. The way we proceed is to generate $K$ different random initializations, apply $N$ iterations of an NMF algorithm to each pair 
and only apply Algorithm~\ref{alg:iterations} to the pair $(W,H)$ with the smallest residual error among those. 
This heuristic can also be found in \cite{CAZP09}. Moreover, note that \ref{alg:ms1} is a particular case of \ref{alg:ms2} with $K = 1$. 
\renewcommand{\thealgorithm}{MS2} 
\begin{algorithm}[ht!] 
	\caption{Multi-Start~2$(X,r,\alpha,\Delta t,K,N)$}
	\label{alg:ms2}
	\begin{algorithmic}[1]
		\REQUIRE $X \in {\mathbb R}^{m\times n}_+$, $r<\min(m,n)$, $0<\alpha<1$, $\Delta t$, $K$, $N$, $\text{tol} = 10 ^{-6}$. 
		\medskip 
		
		 \STATE $e = 1$. 
		\FOR{$i=1 \to K$}
		  \STATE $(\tilde{W},\tilde{H})  \gets$ random initialization$(m,n,r)$. \emph{\% See Section~\ref{rdninit}} 
		  \STATE $[\tilde{W},\tilde{H}] \gets$ AlgoNMF$(X,\tilde{W},\tilde{H},N)$.  \emph{\% See Section~\ref{selectNMFalgo}} 
			\STATE $\tilde{e} = {\|X- \tilde{W}\tilde{H} \|_F}/{\|X\|_F}$. 
			\IF {$\tilde{e} < e$}
		
				\STATE $(W,H) \gets (\tilde{W},\tilde{H})$.
				\STATE $e \gets \tilde{e}$.
				
			\ENDIF 
					\ENDFOR 
		\STATE $[W,H] \gets$ Final Refinement$(X,W,H,\alpha,\Delta t)$. 

	\end{algorithmic}
\end{algorithm}

\subsubsection{Comparing the Multi-Start Heuristics}

Table~\ref{comparMS} gives the computational results for the two multi-start heuristics with different parameters for \ref{alg:ms2} (namely, $N = 20, 40$ and $K = 100,200$). Throughout the paper (unless stated otherwise), the settings are the following: 
\begin{itemize}

\item We use the same randomly generated initial matrices to obtain a fair comparison between the different runs (and for the results to be reproducible). 
In order to do so, we will control the random number generator of Matlab as follows: 
it is initialized with the value 1 (that is, we execute \texttt{rng(1)}) and after each outer loop of the heuristics 
(for example, after step~2 of \ref{alg:ms1} and \ref{alg:ms2}), it is increased by one (that is, we execute \texttt{rng(i+1)} where $i$ is the number of iterations performed so far). 

\item We perform at most 100 runs of each heuristic. In order to reduce the computational time of the numerical experiments, we stop testing a given heuristic as soon as (at least) five exact NMF's for a given nonnegative matrix have been found (this condition being checked after every ten runs). 

\item The tables display the number of exact NMF's found out of the number of runs performed (for example, 6/10 means that the algorithm found six exact NMF's out of ten runs). 
They also display in brackets the average time in seconds needed to compute a single exact NMF. 
The best results in terms of average running time are underlined, 
and the best heuristics in term of robustness (i.e. proportion of exact factorizations found) are in bold; 
see the caption of Table~\ref{comparMS} for more details.   


\end{itemize}

We refer the reader to Section~\ref{selectNMFalgo} for the local search NMF algorithm selected and to Section~\ref{rdninit} for the initialization strategy of matrices $W$ and $H$. 
 All tests are preformed using Matlab on a PC Intel CORE i5-4570 CPU @3.2GHz $\times$ 4 7.7Go RAM. \\

We observe in Table~\ref{comparMS}  that \emph{\ref{alg:ms2} performs better} than \ref{alg:ms1}, while the variants of \ref{alg:ms2} perform similarly. 
Note that the computational times are rather similar: the reason is that the considered matrices are rather small and performing $N$ iterations of the NMF algorithm is therefore relatively quick. 
In the remainder of the paper, we will use the parameters $K = 200$ and $N = 20$ for \ref{alg:ms2} as it offers a good compromise between proportion of exact NMF's found and total computational time.

It is interesting to note that 
\begin{itemize}

\item For randomly generated matrices, as already anticipated in Section~\ref{exactNMFds}, all heuristics are able to identify an exact NMF for all runs. 

\item For some linear EDM's (LEDM12 to LEDM32) and slack matrices (16-G to 24-C), no multi-start strategy is able to identify an exact NMF. 
This observation is the main motivation to develop more efficient heuristics for exact NMF:  
we had to run \ref{alg:ms1} for several hours (which means thousands of initializations) to find an exact NMF of 24-C (slack matrix of the 24-cell). 

\end{itemize}

\begin{table}[h]
\begin{center}
\caption{Comparison of the multi-start heuristics. (The ratio $x/y$ means that $x$ exact NMF's have been found out of $y$ runs of the heuristic, while the number in brackets is the average running time for a heuristic to find a single exact NMF. 
Underlined: (i) the best heuristic in terms of average running time to compute a single exact NMF, and (ii) any heuristic whose running time to compute an exact NMF is at most 10\% away from the best heuristic. 
In bold: (i) the best heuristic in terms of number of exact NMF's found out of a given number of runs, and 
(ii) any heuristic which is at most 10\% away from the best heuristic.  )} 
\begin{tabular}{|c||c||c|c|c|c|}
\hline
 & MS1 & MS2(100,20) & MS2(200,20) & MS2(100,40) & MS2(200,40) \\
\hline
LEDM6 & 5/80 (35) & 9/20 (4.7) & 7/10 (\underline{3.2}) & 11/20 (4.4) & \textbf{9/10} (\underline{3.2}) \\
LEDM8 & 0/100 ($\sim$) & \textbf{6/50} (29.2) & \textbf{6/40} (\underline{20}) & 5/50 (35.8) & \textbf{6/40} (37) \\
LEDM12 & 0/100 ($\sim$) & 0/100 ($\sim$) & 0/100 ($\sim$) & 0/100 ($\sim$) & 0/100 ($\sim$) \\
LEDM16 & 0/100 ($\sim$) & 0/100 ($\sim$) & 0/100 ($\sim$) & 0/100 ($\sim$) & 0/100 ($\sim$) \\
LEDM32 & 0/100 ($\sim$) & 0/100 ($\sim$) & 0/100 ($\sim$) & 0/100 ($\sim$) & 0/100 ($\sim$) \\
6-G & 8/30 (6.6) & \textbf{10/10} (\underline{1.6}) & \textbf{10/10} (1.8) & \textbf{10/10} (1.9) & \textbf{10/10} (2.9) \\
7-G & 8/20 (4.1) & \textbf{9/10} (\underline{1.9}) & \textbf{10/10} (2.2) & \textbf{9/10} (2.3) & \textbf{10/10} (3) \\
8-G & 5/60 (23.3) & 6/10 (3.2) & \textbf{9/10} (\underline{2.3}) & \textbf{9/10} (\underline{2.3}) & \textbf{10/10} (3) \\
9-G & 7/40 (10.6) & 7/10 (\underline{2.8}) & 6/10 (4.2) & 6/10 (4.1) & \textbf{8/10} (4.2) \\
16-G & 0/100 ($\sim$) & 0/100 ($\sim$) & 0/100 ($\sim$) & 0/100 ($\sim$) & 0/100 ($\sim$) \\
32-G & 0/100 ($\sim$) & 0/100 ($\sim$) & 0/100 ($\sim$) & 0/100 ($\sim$) & 0/100 ($\sim$) \\
20-D & 0/100 ($\sim$) & 0/100 ($\sim$) & 0/100 ($\sim$) & 0/100 ($\sim$) & 0/100 ($\sim$) \\
24-C & 0/100 ($\sim$) & 0/100 ($\sim$) & 0/100 ($\sim$) & 0/100 ($\sim$) & 0/100 ($\sim$) \\
UDISJ4 & 6/10 (2.4) & \textbf{10/10} (\underline{1.8}) & \textbf{10/10} (2.1) & \textbf{10/10} (2.2) & \textbf{10/10} (3.4) \\
UDISJ5 & 5/30 (\underline{12.2}) & 5/30 (30.1) & 5/30 (36.4) & \textbf{9/20} (14.9) & \textbf{5/10} (24.2) \\
UDISJ6 & \textbf{2/100} (\underline{119.5}) & 0/100 ($\sim$) & 0/100 ($\sim$) & \textbf{1/100} (1211.9) & 0/100 ($\sim$) \\
RND1 & \textbf{10/10} (\underline{1.1}) & \textbf{10/10} (1.9) & \textbf{10/10} (2.3) & \textbf{10/10} (2.6) & \textbf{10/10} (4.1) \\
RND3 & \textbf{10/10} (\underline{1.1}) & \textbf{10/10} (2) & \textbf{10/10} (2.4) & \textbf{10/10} (2.5) & \textbf{10/10} (4) \\
\hline
\end{tabular}
\label{comparMS}
\end{center}
\end{table}

\subsection{Selecting an Initialization Strategy} \label{rdninit} 

In this section, we describe several random initialization strategies. The most widely used strategy is to generate each entry of the initial $W$ and $H$ factors uniformly at random in the interval [0,1], a strategy which we refer to as RNDCUBE. 
As we will see, RNDCUBE performs rather poorly, and we propose a new very effective random initialization strategy which allows to explore the search domain in a much better way. In fact, the issue with generating each entry of $W$ and $H$ uniformly at random in the interval [0,1] is that it only generates dense matrices, while it is well-known that 
\begin{itemize}
\item [(i)] exact NMF solutions usually have many zero entries (see, e.g., the discussion in \cite{GG09}), and 
\item [(ii)] the boundary of the feasible domain only contains sparse matrices ; hence generating only dense initial matrices starts the exploration relatively far away from that boundary where solutions are in general located. 
\end{itemize}

The sparsest possible way to generate initial matrices with nonzero rank-one factors is the following: we generate $W$ and $H$ so that each column or each row has a single non-zero entry (whose position is chosen at random). This leads to four possible initializations denoted SPARSE$ij$: 
$i=0$ (resp.\@ $j= 0$) means that $W$ (resp.\@ $H$) has a single non-zero entry by row, and 
$i=1$ (resp.\@ $j= 1$) means that $W$ (resp.\@ $H$) has a single non-zero entry by column.  


Table~\ref{tableinit} reports the numerical results. 
\begin{table}[h]
\footnotesize
\begin{center}
\caption{Comparison of the different initialization strategies combined with multi-start~2.}
\begin{tabular}{|c||c|c|c|c|c|}
\hline
 & sparse 00 & sparse 10 & sparse 01 & sparse 11 & rndcube \\
\hline
LEDM6 & 5/100 (54.8) & 5/90 (50.6) & \textbf{8/10} (\underline{2.7}) & 7/10 (3.2) & 6/60 (25.1) \\
LEDM8 & 4/100 (113.3) & 3/100 (133.5) & \textbf{6/30} (23.5) & \textbf{6/40} (\underline{20}) & 0/100 ($\sim$) \\
LEDM12 & 0/100 ($\sim$) & 0/100 ($\sim$) & 0/100 ($\sim$) & 0/100 ($\sim$) & 0/100 ($\sim$) \\
LEDM16 & 0/100 ($\sim$) & 0/100 ($\sim$) & 0/100 ($\sim$) & 0/100 ($\sim$) & 0/100 ($\sim$) \\
LEDM32 & 0/100 ($\sim$) & 0/100 ($\sim$) & 0/100 ($\sim$) & 0/100 ($\sim$) & 0/100 ($\sim$) \\
6-G & \textbf{10/10} (\underline{1.7}) & \textbf{10/10} (1.9) & \textbf{10/10} (2.1) & \textbf{10/10} (\underline{1.8}) & \textbf{10/10} (2) \\
7-G & \textbf{10/10} (\underline{1.9}) & 7/10 (2.9) & \textbf{10/10} (2.2) & \textbf{10/10} (2.2) & \textbf{10/10} (\underline{1.9}) \\
8-G & 6/10 (4.1) & 6/10 (3.8) & \textbf{9/10} (\underline{2.5}) & \textbf{9/10} (\underline{2.3}) & 5/30 (16.7) \\
9-G & 8/20 (6.3) & 5/30 (16.1) & 5/10 (4.9) & \textbf{6/10} (\underline{4.2}) & 5/40 (23.3) \\
16-G & 0/100 ($\sim$) & 0/100 ($\sim$) & 0/100 ($\sim$) & 0/100 ($\sim$) & 0/100 ($\sim$) \\
32-G & 0/100 ($\sim$) & 0/100 ($\sim$) & 0/100 ($\sim$) & 0/100 ($\sim$) & 0/100 ($\sim$) \\
20-D & 0/100 ($\sim$) & 0/100 ($\sim$) & 0/100 ($\sim$) & 0/100 ($\sim$) & 0/100 ($\sim$) \\
24-C & 0/100 ($\sim$) & 0/100 ($\sim$) & 0/100 ($\sim$) & 0/100 ($\sim$) & 0/100 ($\sim$) \\
UDISJ4 & \textbf{10/10} (\underline{2.2}) & \textbf{10/10} (\underline{2.1}) & \textbf{10/10} (\underline{2.3}) & \textbf{10/10} (\underline{2.1}) & \textbf{10/10} (\underline{2.1}) \\
UDISJ5 & \textbf{7/20} (\underline{17.1}) & 6/20 (20.4) & 6/30 (30.5) & 5/30 (36.4) & \textbf{7/20} (\underline{16.5}) \\
UDISJ6 & \textbf{2/100} (465.9) & \textbf{5/80} (\underline{146.3}) & \textbf{1/100} (921.8) & 0/100 ($\sim$) & 0/100 ($\sim$) \\
RND1 & \textbf{10/10} (\underline{2.3}) & \textbf{10/10} (2.4) & \textbf{10/10} (2.8) & \textbf{10/10} (\underline{2.3}) & \textbf{10/10} (\underline{2.1}) \\
RND3 & \textbf{10/10} (2.5) & \textbf{10/10} (\underline{2.3}) & \textbf{10/10} (\underline{2.2}) & \textbf{10/10} (\underline{2.4}) & \textbf{10/10} (\underline{2.3}) \\
\hline
\end{tabular}
\label{tableinit}
\end{center}
\end{table} 
As explained above, RNDCUBE does not perform as well as the sparse initialization strategies (for example, it is not able to find an exact NMF of LEDM8 while all other initialization strategies are). 
\emph{SPARSE11 has on average the best results} and we will therefore select it as the initialization strategy for \ref{alg:ms2} for the remainder of the paper.


\subsection{Selecting an NMF Algorithm} \label{selectNMFalgo}

In order to design heuristics for exact NMF, 
a local search heuristic is needed to improve a given solution (i.e.\@ pair of factors $W$ and $H$) locally. 
Most NMF algorithms could potentially be used: in fact, most NMF algorithms are local search heuristic based on standard nonlinear optimization schemes.  
In this section, we compare the following state-of-the-art NMF algorithms in order to assess their performances for computing exact NMF's: 
\begin{enumerate} 
\item (\textbf{MU}) \, The multiplicative updates (MU) algorithm  of \cite{LS99, LS2}. 
\item (\textbf{A-MU}) \, The accelerated MU from \cite{GG12}.  
\item (\textbf{HALS}) \, The hierarchical alternating least squares (HALS) algorithm from \cite{CZA07, CP09b}.  
\item (\textbf{A-HALS}) \, The accelerated HALS from \cite{GG12}. 
\item (\textbf{ANLS}) \, The alternating nonnegative least squares algorithm of \cite{KP11}, which alternatively optimizes $W$ and $H$ exactly using a  block-pivot active set method; see also \cite{KHP14}. 
\end{enumerate} 
The code of the first four algorithms is available at \url{https://sites.google.com/site/nicolasgillis/}. 
The code of ANLS was obtained from \url{http://www.cc.gatech.edu/~hpark/}.

The convergence speeds of these NMF algorithms were previously compared on real-world image and document data sets, and A-HALS was shown to perform the best in most cases. 
However, in this paper, we are interested in finding exact NMF's of relatively small matrices.  
Our goal in this section is therefore to identify which algorithm is the best at identifying exact NMF's of such matrices when used as a subroutine for \ref{alg:ms2}; see Table~\ref{comparNMFalgo}. 
HALS and A-HALS perform on average the best in terms of number of exact NMF's found (note that A-HALS is not much faster than HALS because the parameter $\Delta t$ was set to a rather large value, hence both algorithms are able to converge within the alloted time). 
ANLS performs rather poorly because it runs into numerical problems for rank-deficient factors $W$ (and/or $H$), 
which appear as solutions of exact NMF's of nonnegative matrices $X$ with $\rank_+(X) > \rank(X)$ \cite{GG10b}. MU and A-MU also perform poorly: because of their multiplicative nature, they cannot deal very well with sparse solutions\footnote{Note that we used the variants of MU and A-MU proposed \cite{GG12} where zero entries of $W$ and $H$ are replaced with a small positive number (we used $10^{-16}$) so that they can modify zero entries, and a subsequence is guaranteed to converge to a stationary point \cite{TR14}.}; 
see, e.g., the discussion in \cite{GG12}.

\begin{table}[h]
\begin{center}
\caption{Comparison of NMF algorithms combined with multi-start~2.}
\begin{tabular}{|c||c|c|c|c|c|}
\hline
 & ANLS & MU & A-MU & HALS & A-HALS \\
\hline
LEDM6 & 0/100 ($\sim$) & 0/100 ($\sim$) & 0/100 ($\sim$) & \textbf{8/10} (\underline{2.8}) & 7/10 (3.2) \\
LEDM8 & 0/100 ($\sim$) & 0/100 ($\sim$) & 0/100 ($\sim$) & \textbf{5/30} (\underline{20.8}) & \textbf{6/40} (\underline{20}) \\
LEDM12 & 0/100 ($\sim$) & 0/100 ($\sim$) & 0/100 ($\sim$) & 0/100 ($\sim$) & 0/100 ($\sim$) \\
LEDM16 & 0/100 ($\sim$) & 0/100 ($\sim$) & 0/100 ($\sim$) & 0/100 ($\sim$) & 0/100 ($\sim$) \\
LEDM32 & 0/100 ($\sim$) & 0/100 ($\sim$) & 0/100 ($\sim$) & 0/100 ($\sim$) & 0/100 ($\sim$) \\
6-G & 0/100 ($\sim$) & 0/100 ($\sim$) & 0/100 ($\sim$) & \textbf{10/10} (2.1) & \textbf{10/10} (\underline{1.8}) \\
7-G & 0/100 ($\sim$) & 0/100 ($\sim$) & 0/100 ($\sim$) & \textbf{10/10} (\underline{2.1}) & \textbf{10/10} (\underline{2.2}) \\
8-G & 0/100 ($\sim$) & 0/100 ($\sim$) & 0/100 ($\sim$) & \textbf{9/10} (\underline{2.4}) & \textbf{9/10} (\underline{2.3}) \\
9-G & 0/100 ($\sim$) & 1/100 (405.3) & 5/70 (134.2) & 5/10 (5.4) & \textbf{6/10} (\underline{4.2}) \\
16-G & 0/100 ($\sim$) & 0/100 ($\sim$) & 0/100 ($\sim$) & 0/100 ($\sim$) & 0/100 ($\sim$) \\
32-G & 0/100 ($\sim$) & 0/100 ($\sim$) & 0/100 ($\sim$) & 0/100 ($\sim$) & 0/100 ($\sim$) \\
20-D & 0/100 ($\sim$) & 0/100 ($\sim$) & 0/100 ($\sim$) & 0/100 ($\sim$) & 0/100 ($\sim$) \\
24-C & 0/100 ($\sim$) & 0/100 ($\sim$) & 0/100 ($\sim$) & 0/100 ($\sim$) & 0/100 ($\sim$) \\
UDISJ4 & \textbf{10/10} (13) & 0/100 ($\sim$) & 0/100 ($\sim$) & \textbf{10/10} (2.4) & \textbf{10/10} (\underline{2.1}) \\
UDISJ5 & 5/100 (778.7) & 0/100 ($\sim$) & 0/100 ($\sim$) & \textbf{5/40} (58.5) & \textbf{5/30} (\underline{36.4}) \\
UDISJ6 & 0/100 ($\sim$) & 0/100 ($\sim$) & 0/100 ($\sim$) & \textbf{1/100} (\underline{1105.4}) & 0/100 ($\sim$) \\
RND1 & 5/20 (71.6) & 3/100 (472.5) & 5/10 (33.2) & \textbf{10/10} (2.8) & \textbf{10/10} (\underline{2.3}) \\
RND3 & \textbf{10/10} (27.5) & 8/20 (161.6) & 7/10 (49.1) & \textbf{10/10} (2.8) & \textbf{10/10} (\underline{2.4}) \\
\hline
\end{tabular}
\label{comparNMFalgo}
\end{center}
\end{table}


In light of these results, \emph{we select A-HALS} as the NMF algorithm for the remainder of the paper.

\section{Two Heuristics for NMF} \label{sagrasp}

In this section, we propose two heuristics for NMF, along with a hybridization strategy. 

\subsection{Simulated Annealing} \label{saheur}

The first heuristic we propose follows the widely used simulated annealing framework \cite{Pirlo96}; see Algorithm~\ref{alg:sa} which we briefly describe here. 
As for the multi-start heuristics, \ref{alg:sa} first generates an initial solution $(W,H)$. 
\ref{alg:sa} will then explore the neighborhood of this initial solution in a random fashion in the hope to find a better solution. 
 A solution in the neighborhood will be computed by repeating $K$ times the following steps: 
\begin{itemize}

\item select a small subset $\mathcal{J}$ of $J$ rank-one factors $W(:,\mathcal{J})H(\mathcal{J},:)$ at random, that is, generate randomly $\mathcal{J} \subset \{1,2,\dots,r\}$ such that $|\mathcal{J}| = J$, 

\item reinitialize these rank-one factors randomly (see Section~\ref{rdninit}),  

\item improve the corresponding solution locally (we will use $N$ iterations of A-HALS; see Section~\ref{selectNMFalgo}), and 

\item decide whether to keep the refined neighboring solution depending on its error and on the current temperature, see step~14 of Algorithm~\ref{alg:sa} (the higher the temperature, the more likely it is for a solution to be accepted as the next iterate). Note that a solution whose error is smaller than the error of the current solution is always kept.  Hence an important characteristic of \ref{alg:sa} is that it allows for solutions with higher errors to be explored (although the probability for this to happen goes to zero as the temperature decreases). 

\end{itemize} 

The procedure is repeated several times for several temperatures 
(from $T_0$ to $T_{end}$ with 20 loga-rithmically-spaced intermediate values 
). 
We use the following values for the parameters: initialization SPARSE10, $T_{0} = 0.1$ for the initial temperature (this means for example that the initial temperature allows for a solution with relative error 10\% higher than the current solution to be accepted with probability $e^{-1} \approx 1/3$),  $T_{end} = 10^{-4}$ for the final temperature (this means for example that the final temperature allows for a solution with relative error $0.1\%$ higher than the current solution to be accepted with probability $e^{-1} \approx 1/3$), 
$J=2$, $N = 100$ and $K = 50$; see Appendix~\ref{appB} for numerical experiments for different values of the parameters.  It is important to point out that the initialization procedure is crucial: in fact, SPARSE10 allows to compute exact NMF for all considered matrices while RNDCUBE fails to do so and is much slower to compute exact NMF's; see Appendix~\ref{appB}.  
\renewcommand{\thealgorithm}{SA} 
\begin{algorithm}
	\caption{Simulated Annealing$(X,r,\alpha,\Delta t, T_0, T_{end}, \beta, K, N, J, \text{tol})$}
	\label{alg:sa}
	\begin{algorithmic}[1]
		\REQUIRE $X \in {\mathbb R}^{m\times n}_+$, $r < \min(m,n)$, $0<\alpha<1$, $T_0$, $T_{end}$, $0<\beta<1$, $K$, $N$, $J$, tol. 
		
		\STATE $(W,H) \gets$ random initialization$(m,n,r)$. \emph{\% See Section~\ref{rdninit}} 
		\STATE $e \gets \frac{\|X-WH\|_F}{\|S\|_F}$.
		\STATE $e_{\min} \gets e$.
		\STATE $T \gets T_0$
		\WHILE{$T>T_{end}$}
			\FOR{$i=1 \to K$}
				\STATE $(\tilde{W}, \tilde{H}) \gets ({W}, {H})$. 
				\STATE $\mathcal{J} \gets$ pick randomly $J$ indices in $\{1,2,\dots,r\}$. 
				\STATE $(\tilde{W}(:,\mathcal{J}), \tilde{H}(\mathcal{J},:)) \gets$ random initialization$(m,n,J)$.  
				\STATE $[\tilde{W},\tilde{H}] \gets$ AlgoNMF$(X,\tilde{W},\tilde{H},N)$. 
				\STATE $\tilde{e} \gets \frac{\|X-\tilde{W}\tilde{H}\|_F}{\|X\|_F}$.
				\STATE $\Delta \gets \tilde{e}-e$.
				\STATE  \emph{\% $U[0,1]$ is the uniform distribution in $[0,1]$ (\texttt{rand} in Matlab}) 
				\IF{$U[0,1] < \exp\left(-\frac{\Delta}{T}\right)$} 
					\STATE $W \gets \tilde{W}$, $H \gets \tilde{H}$, $e \gets \tilde{e}$.
					\IF{$e < e_{\min}$}
						\STATE $e_{\min} \gets e$.
						\STATE $(W_{\min}, H_{\min}) \gets (\tilde{W}, \tilde{H})$.
					\ENDIF
					\IF{$e_{\min} < \text{tol}$}
						\STATE $T = T_{end}$; break. 
					\ENDIF
				\ENDIF
			\ENDFOR
			\STATE $T \gets \beta T$.
		\ENDWHILE  
		\STATE Return $[W_{\min},H_{\min}]$.  
	\end{algorithmic}
\end{algorithm}

\subsection{Rank-by-Rank Heuristic} \label{rbrheur}

The second heuristic tries to construct recursively an exact NMF $(W,H)$ of $X$ as follows (see Algorithm~\ref{alg:rbr}): 
\begin{itemize}

\item at the first step ($k = 1$), an optimal rank-one NMF $(W_1,H_1)$ of $X$ is computed. This can be done for example using the truncated singular value decomposition using the Perron-Frobenius and Eckart-Young theorems. 

\item At the $k$th step ($2 \leq k \leq r$), a rank-$k$ NMF solution is generated combining  the rank-$(k-1)$ NMF solution $(W_{k-1},H_{k-1})$ computed at the $(k-1)$th step with an additional rank-one factor randomly generated. 
This solution is then locally improve using $N$ steps of an NMF algorithm. 
This procedure is repeated $K$ times and the best solution is kept; see Algorithm~\ref{alg:grpo}. 

\end{itemize} 
RBR will turn out to be a powerful exact NMF heuristic for some classes of matrices 
(such as slack matrices).   
\renewcommand{\thealgorithm}{RBR} 
\begin{algorithm}[h]
	\caption{Rank-by-Rank Heuristic$(X,r,\alpha,\Delta t,K,N)$}
	\label{alg:rbr}
	\begin{algorithmic}[1]
		\REQUIRE $X \in {\mathbb R}^{m\times n}_+$, $r<\min(m,n)$, $0<\alpha<1$, $\Delta t$, $K$, $N$. 
		\STATE $[w_1,\sigma_1,h_1] \gets$ svds$(X,1)$. \emph{\% See \texttt{svds} function of Matlab} 
		\STATE $(W_1, H_1) \gets \left( \left|w_1\right|, \sigma_1 \left|h_1^T \right| \right)$ \emph{\% This is an optimal nonnegative rank-one approximation of $X$}  
		\FOR{$k=2 \to r$}
			\STATE $[W_k,H_k] \gets$ getRankPlusOne$(X,W_{k-1},H_{k-1},K,N)$.
		\ENDFOR
	\end{algorithmic}
\end{algorithm}

\renewcommand{\thealgorithm}{getRankPlusOne} 
\begin{algorithm}[h]
	\caption{getRankPlusOne$(X,W,H,K,N)$} \label{alg:grpo} 
	\label{algo_getRankK}
	\begin{algorithmic}[1]
		\REQUIRE $X \in {\mathbb R}^{m\times n}_+$, $W \in {\mathbb R}^{m\times k-1}_+$, $H \in {\mathbb R}^{k-1\times n}_+$, $K$, $N$.
		\STATE $e_{\min} \gets 1$. 
		\FOR{$j=1 \to K$}
			\STATE $\left(\tilde{W}(:,1:k-1),\tilde{H}(1:k-1,:)\right) \gets ({W},{H})$. 
			\STATE $\left(\tilde{W}(:,k), \tilde{H}(k,:)\right) \gets$ random initialization$(m,n,1)$. 
			\STATE $[\tilde{W},\tilde{H}] \gets$ AlgoNMF$(X,\tilde{W},\tilde{H},N)$. 
			\STATE $\tilde{e} \gets \frac{\|X-\tilde{W}\tilde{H}\|_F}{\|X\|_F}$
			\IF{$\tilde{e} < e_{\min}$}
					\STATE $e_{\min} \gets \tilde{e}$, $W_{\min} \gets \tilde{W}$, $H_{\min} \gets \tilde{H}$. 
				\ENDIF
		\ENDFOR
		\STATE Return $[W_{\min},H_{\min}]$. 
	\end{algorithmic}
\end{algorithm}

We will use SPARSE10 for the initialization, $K = 10$ and $N = 50$ which seem to be a good compromise; see Appendix~\ref{appC} for tests of differents values.

\subsection{Hybridization} \label{hybrid}

When designing heuristics, a standard technique consists in using hybridization, that is, to combine several heuristics. 
For example, instead of refining the solution computed by \ref{alg:rbr} with the final refinement step, it is possible to call Simulated Annealing instead; in other words, we propose to initialize \ref{alg:sa} with \ref{alg:rbr}. We refer to this heuristic as `Hybrid'.

\section{Numerical Experiments: Comparing Exact NMF Heuristics} \label{ne}

In this section, we compare 
\ref{alg:ms1}, 
\ref{alg:ms2}, 
\ref{alg:sa}, 
\ref{alg:rbr} and Hybrid, 
with a maximimum number of 1000 runs, and stop the execution of an heuristic when 100 exact NMF's were found (checking this condition every 50 runs); see Table~\ref{finaltable}. 
\begin{table}[h]
\footnotesize 
\begin{center}
\caption{Comparison of the different heuristics:
\ref{alg:ms1} and \ref{alg:ms2} (Section~\ref{msstrategies}), 
\ref{alg:sa} (Section~\ref{saheur}),
\ref{alg:rbr} (Section~\ref{rbrheur}) and Hybrid (Section~\ref{hybrid}).}
\begin{tabular}{|c||c|c|c|c|c|}
\hline
 & MS1 & MS2 & SA & RBR & Hybrid \\
\hline
LEDM6 & 40/1000 (53.5) & 112/150 (3.1) & \textbf{100/100} (19.6) & \textbf{100/100} (\underline{1.4}) & \textbf{100/100} (19) \\
LEDM8 & 0/1000 ($\sim$) & 107/600 (27.1) & \textbf{100/100} (60.9) & \textbf{100/100} (\underline{16.7}) & \textbf{148/150} (63.6) \\
LEDM12 & 0/1000 ($\sim$) & 0/1000 ($\sim$) & 119/200 (42.9) & 107/650 (\underline{15.1}) & \textbf{103/150} (36.9) \\
LEDM16 & 0/1000 ($\sim$) & 0/1000 ($\sim$) & 100/250 (118.3) & 100/550 (\underline{29.1}) & \textbf{121/250} (104.2) \\
LEDM32 & 0/1000 ($\sim$) & 0/1000 ($\sim$) & \textbf{14/1000} (2592.9) & 0/1000 ($\sim$) & \textbf{28/1000} (\underline{1370.9}) \\
6-G & 108/700 (12.1) & \textbf{100/100} (2.1) & \textbf{100/100} (\underline{1.2}) & \textbf{100/100} (1.4) & \textbf{100/100} (1.5) \\
7-G & 104/350 (5.8) & \textbf{100/100} (2.2) & \textbf{100/100} (4.2) & \textbf{100/100} (\underline{1.5}) & \textbf{100/100} (4.4) \\
8-G & 61/1000 (32.2) & 129/200 (3.8) & \textbf{100/100} (15.4) & \textbf{100/100} (\underline{1.5}) & \textbf{100/100} (15.3) \\
9-G & 104/700 (12.8) & 117/200 (4.6) & \textbf{100/100} (22.9) & \textbf{100/100} (\underline{1.6}) & \textbf{100/100} (23.2) \\
16-G & 0/1000 ($\sim$) & 0/1000 ($\sim$) & 102/350 (91.6) & \textbf{143/150} (\underline{1.9}) & 118/150 (34.2) \\
32-G & 0/1000 ($\sim$) & 0/1000 ($\sim$) & 31/1000 (1086.8) & \textbf{107/250} (\underline{6.6}) & 105/300 (97) \\
20-D & 1/1000 (2021.1) & 21/1000 (160.9) & \textbf{100/100} (7.8) & \textbf{129/150} (\underline{2.3}) & \textbf{100/100} (5.6) \\
24-C & 0/1000 ($\sim$) & 0/1000 ($\sim$) & \textbf{100/100} (\underline{3.1}) & 119/200 (4.1) & \textbf{100/100} (4.4) \\
UDISJ4 & 102/250 (4) & \textbf{100/100} (2.4) & \textbf{100/100} (\underline{1.2}) & \textbf{100/100} (1.9) & \textbf{100/100} (1.9) \\
UDISJ5 & 104/850 (17.4) & 102/500 (38) & \textbf{100/100} (\underline{2.8}) & \textbf{100/100} (4.9) & \textbf{100/100} (5.2) \\
UDISJ6 & 7/1000 (337.1) & 8/1000 (1594.7) & \textbf{100/100} (\underline{7.8}) & 112/450 (66.4) & \textbf{100/100} (18.5) \\
RND1 & \textbf{148/150} (\underline{1.1}) & \textbf{100/100} (2.8) & \textbf{100/100} (\underline{1.1}) & \textbf{100/100} (2.2) & \textbf{100/100} (2.2) \\
RND3 & \textbf{100/100} (\underline{1.1}) & \textbf{100/100} (2.8) & \textbf{100/100} (\underline{1.1}) & \textbf{100/100} (2.2) & \textbf{100/100} (2.2) \\
\hline
\end{tabular}
\label{finaltable}
\end{center}
\end{table} 

As already pointed out, the multi-start heuristics perform rather poorly and are not able to compute even a single exact NMF in many cases.  
We observe that 
\begin{itemize}

\item \ref{alg:rbr} is able to compute an exact NMF for all matrices but LEDM32, while \ref{alg:sa} and Hybrid are able to find an exact NMF for all matrices. 

\item In terms of robustness, Hybrid is the best as it is able to compute on average the most exact NMF's for a fixed number of runs. 

\item In terms of running times, \ref{alg:rbr} is on average the fastest, while Hybrid is comparatively much slower. 

\end{itemize}

Therefore, in practice, we would recommend to first run RBR as it computes, in many cases, exact NMF's the fastest. 
Moreover, for some matrices (e.g., slack matrices of regular $n$-gons), it is very robust. Then, when RBR fails, we would recommend to run Hybrid because of its robustness: although it is slower, it is in general more likely to find exact NMF's.

\subsection{Limits of the Heuristics for Exact NMF}  \label{limits}

Computing exact NMF's becomes more challenging when the dimensions and the nonnegative rank of the matrix increases, as the computational complexity of the problem depends on these dimensions (see the discussion in Section~\ref{ccomp}). To illustrate the limitations of the use of heuristics to find exact NMF's, 
Table~\ref{tablelarge} reports the computational results for larger slack matrices of regular $n$-gons. \\

\begin{table}[h]
\begin{center}
\caption{}
\begin{tabular}{|c||c|}
\hline
 & Hybrid \\
\hline
110-G & \textbf{14/1000} (\underline{12050.3}) \\
120-G & \textbf{12/1000} (\underline{15556.4}) \\
130-G & \textbf{12/1000} (\underline{16462.7}) \\
140-G & \textbf{15/1000} (\underline{14002.6}) \\
150-G & \textbf{5/1000} (\underline{49277}) \\
160-G & \textbf{1/1000} (\underline{144803.1}) \\
170-G & 0/1000 ($\sim$) \\
\hline
\end{tabular}
\label{tablelarge}
\end{center}
\end{table}

Moreover, for LEDM of size 48-by-48 or larger, and for slack matrices of regular $n$-gons with $n \geq 170$, 
none of our heuristics is able to find a single exact NMF's out of 1000 runs.

\section{Using Exact NMF Heuristics for New Insights on the Nonnegative Rank} \label{conject}

In this section, we describe four important open questions related to the nonnegative rank, and show how computing exact NMF's of small matrices can help gain insights about them.

\subsection{Kronecker Product of Two Nonnegative Matrices} \label{krone}

In a recent Dagsthul seminar \cite{BLHT13}, 
participants came up with the following conjecture: given two nonnegative matrices $A$ and $B$, the nonnegative rank of their Kronecker product is equal to the product of their nonnegative ranks, that is, 
\[
\rank_+(A \otimes B) = \rank_+(A) \rank_+(B). 
\]
Note that this results holds for the usual rank, and that it is easy to show that $\rank_+(A \otimes B) \leq \rank_+(A) \rank_+(B)$ (see also \cite{Wat14} for a short discussion). Hamza Fawzy used the multi-start strategy MS1 to come up with the following counter example: 
\[ 
A = \left( \begin{array}{cccc}
    1  &  0 &  1  &  a\\
    0  &  1 &  0  &  1-a\\
    0  &  0 &  1  &  1-a\\
    1  &  1 &  0  &  a \\
    \end{array}
\right),
\]
where $a = 3/8$ from \cite{BCR11} for which $\rank_+(A) = 4$ and $\rank_+(A \otimes A) = 15$. One may therefore 
wonder whether the following is true 
\[
\rank_+(A \otimes B) \geq \rank_+(A) \rank_+(B) - 1 \quad ? 
\]
It turns out that it is also incorrect. 
In fact, we have found a 4-by-4 nonnegative matrix $A$ with rank 3 and nonnegative rank 4 such that $\rank_+(A \otimes A)= 12$ : 
\[ 
A = \left( \begin{array}{cccc}
    1+a  &  1-a  &  1-a  &  1+a\\
    1-a  &  1+a  &  1+a  &  1-a\\
    1+a   & 1+a  &  1-a  &  1-a\\
    1-a  &  1-a &   1+a  &  1+a \\
    \end{array}
\right), 
\]
where $a = \sqrt{2}-0.9$. Geometrically, the matrix $A$ is the (generalized) slack matrix of a pair of polytopes, namely two nested squares: the rows of $A$ correspond to the edges of the outer square and its columns to the vertices of the inner square; see \cite{G12} for more details. 
For $\sqrt{2}-1 < a \leq 1$, $\rank_+(A) = 4$. However, for $a = 1$ (which corresponds to the slack matrix of the square, that is, the regular 4-gon), $\rank_+(A \otimes A)= 16$. Decreasing $a$ sufficiently while keepig $a > \sqrt{2}-1$ allows to decrease $\rank_+(A \otimes A)$ to 12 while keeping $\rank_+(A) = 4$. 
The intuition behind this example is the following: decreasing $a$ leaves more space between the two squares although no triangle fits between the two (hence $\rank_+(A) = 4$). However, this makes the search space of the exact NMF problem for $A \otimes A$ much larger, leading to the existence of an exact NMF with smaller rank.

How the nonnegative rank of the Kronecker product between two matrices relates with their nonnegative ranks remains an open question.
This is an important open question, and, as illustrated above, exact NMF algorithms are useful tools to address such questions. 
In light of the above example, a new conjecture could be the following:  
\begin{conj} \label{conj1}
For any nonnegative matrix $A$, 
\[
\rank_+(A \otimes A) \geq \rank_+(A) \rank(A).  
\]
\end{conj}

\subsection{Slack Matrices of Regular $n$-gons}

As explained in the introduction, the nonnegative rank of the slack matrix $X_n$ of the regular $n$-gon is equal to its extension complexity, that is, to the minimum number of facets a higher dimensional polytope requires to represent it after a linear projection.   
It can be shown that $\rank_+(X_n) \geq \left\lceil \log_2(2n+2) \right\rceil$ \cite{G09}. 
Ben-Tal and Nemirovski \cite{BTN01} gave an extension of regular $n$-gons when $n$ is a power of two ($n = 2^k$ for some $k$) with $2 \log_2(n) + 4$ facets. 
They used this construction to approximate the circle with regular $n$-gons which allowed them to approximate second-order cone programs with linear programs. Another construction for arbitrary $n$ was proposed in \cite{FRT12} 
showing that $\rank_+(X_n) \leq \left\lceil 2 \log_2(n) \right\rceil$. 
However, the exact value of $\rank_+(X_n)$ is unknown (except for small $n$; see below). 

We have run the Hybrid heuristic on these matrices for all $n \leq 78$ and observe the following:  
\begin{conj} \label{conj2}
The nonnegative rank of the slack matrix $X_n$ of the regular $n$-gon is given by  
\[
    \rank_+(X_n) = 
		 \left\{ \begin{array}{ccccc} 
		2k-1 & \text{for } & 2^{k-1}       & < \quad n  \quad \leq & 2^{k-1}+2^{k-2}, \\ 
		2k   & \text{for } & 2^{k-1}+2^{k-2} & <  \quad n   \quad\leq &  2^{k} . 
  \end{array} \right. 
\] 
\end{conj}
Note that the conjecture is known to be true for $n \leq 9$ as it matches a lower bound based on the rectangle covering bound improved with additional rank constraints from \cite{ArnaudStefan14}.  

For all slack matrices with\footnote{Because it requires a rather high computational cost for larger $n$, we stopped testing the conjecture at $n=78$. In fact, running this experiment on a regular laptop took about two weeks.} $3 \leq n \leq 78$, Hybrid was able to compute at least one exact NMF matching the nonnegative rank given in Conjecture~\ref{conj2}, while it was never able to compute a single exact NMF with a smaller nonnegative rank (out of 1000 runs). 
 Figure~\ref{regngons} displays the number of exact NMF's found out of 1000 initializations of Hybrid for the nonnegative rank given in Conjecture~\ref{conj2}. 
\begin{figure}[ht!] 
\centering 
   \includegraphics[width=9cm]{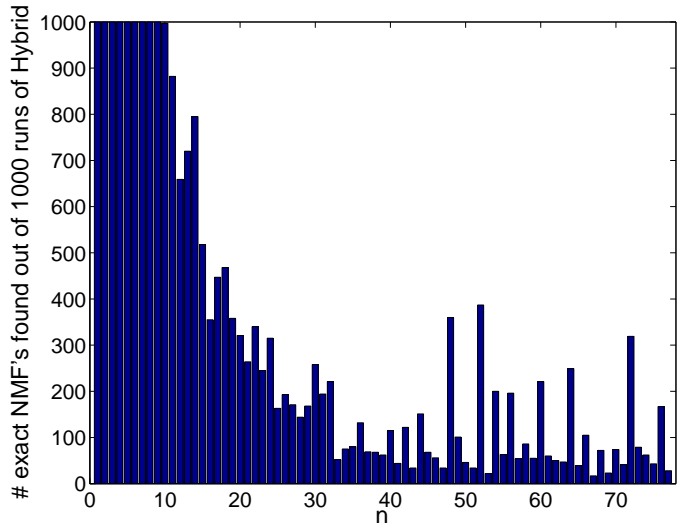} 
\caption{Number of exact NMF found out of 1000 runs of Hybrid on regular $n$-gons and for the nonnegative rank given in Conjecture~\ref{conj2}.}
\label{regngons}
\end{figure}

It is interesting and quite natural to observe that, as $n$ increases, Hybrid meets more and more difficulty to compute an exact NMF of these slack matrices. This illustrates the limits of heuristics to solve exact NMF problems for larger (and difficult) matrices; see also Section~\ref{limits}.

\subsection{Generic $n$-gons}

A generic $n$-gon is an $n$-gon for which the vertices were generated randomly\footnote{The exact definition given in \cite{FRT12} is the following: `a polygon in $\mathbb{R}^2$ is generic if the coordinates of its vertices are distinct and form a set that is algebraically independent over the rationals'.}. 
It is known that the nonnegative rank of the slack matrices $X_n$ of generic $n$-gons satisfy 
\[
\sqrt{2n} 
\quad \leq  \quad  
\rank_+(X_n) 
\quad \leq  \quad  
 \left\lceil \frac{6n}{7} \right\rceil. 
\]
The lower bound is due to \cite{FRT12}, while the upper bound is generic for any $n$-gon and is due to Shitov~\cite{Shit14} (it was also proved using different arguments in \cite{PP14}).  
An important question is to characterize the growth of the the nonnegative rank of these slack matrices: is it proportional to $\sqrt{n}$, $n$ or something in between \cite{Gouveia} ? 

As $n$ increases, it becomes more and more difficult to generate generic $n$-gons (because it is likely that a newly generated point belongs to the convex hull of the previously generated points). 
Therefore we used the following procedure. 
We generate random $n$-gons whose vertices lie on the unit circle. 
To obtain polygons whose vertices are relatively well separated form the convex hull generated by 
the other vertices\footnote{As a vertex gets closer and closer to the convex hull generated by the other vertices, it becomes numerically harder and harder to decide whether or not it belongs to the convex hull.}, 
we subdivide the circle into $n$ disjoint arcs of the same length. Then, each arc is divided into four parts of the same length and we only generate one point randomly into the two middle parts (uniformly distributed). This ensures the angles between any two data points to be larger than $\frac{\pi}{n}$. Then, for each $n$, we generate ten such random $n$-gons and run Hybrid with 1000 runs. Table~\ref{tableirregular} reports the minimum and maximum number of exact NMF's found among these ten matrices. 
\begin{sidewaystable}[h]
\footnotesize
\begin{center}
\caption{Nonnegative rank of random $n$-gons on the circle: for a given $n$, the table reports the minimum and maximum number of exact NMF's found by Hybrid out of 1000 runs on ten such $n$-gons.}
\begin{tabular}{|c||c|c|c|c|c|c|c|c|c|c|c|c|}
\hline
r/n & $6$ & $7$ & $8$ & $9$ & $10$ & $11$ & $12$ & $13$ & $14$ & $15$ & $16$ & $17$ \\
\hline
$4$ &  & & & & & & & & & & &\\
$5$ & [0,0] &[0,0] & & & & & & & & & &\\
$6$ & [1000,1000] &[463,1000] &[0,0] &[0,0] & & & & & & & &\\
$7$ &  & &[754,897] &[160,353] &[0,0] &[0,0] & & & & & &\\
$8$ &  & & &[743,873] &[351,443] &[25,48] &[0,0] &[0,0] & & & &\\
$9$ &  & & & &[787,858] &[401,546] &[148,190] &[10,19] &[0,0] &[0,0] &[0,0] &\\
$10$ &  & & & & &[692,862] &[580,665] &[242,389] &[63,111] &[5,19] &[0,1] &[0,0]\\
$11$ &  & & & & & &[833,902] &[533,726] &[385,540] &[150,247] &[9,82] &[5,14]\\
$12$ &  & & & & & & &[734,874] &[643,766] &[442,631] &[138,365] &[107,204]\\
$13$ &  & & & & & & & & &[671,824] &[375,674] &[405,517]\\
$14$ &  & & & & & & & & & &[610,830] &[583,734]\\
$15$ &  & & & & & & & & & & &[721,829]\\
\hline
\end{tabular}
\label{tableirregular}
\end{center}
\end{sidewaystable}

These results suggests for example that generic 12-gons have extension complexity equal to 9 --which also suggests that all 12-gons have extension complexity smaller than $9$. 
More generally, these results lead us to the following conjecture 
\begin{conj} \label{conjsm}
The nonnegative rank of the slack matrix $S_n$ of any $n$-gon is bounded above 
by $\left\lfloor \frac{n+6}{2} \right\rfloor$ where $\left\lfloor x \right\rfloor$ is the largest integer smaller than $x$, that is, 
\[
\rank_+(S_n) \; \leq \; \left\lfloor \frac{n+6}{2} \right\rfloor , 
\]
and equality holds for $5 \leq n \leq 15$. 
\end{conj} 
Another open question is the following: \emph{For $n$ fixed, are the nonnegative ranks of the slack matrices of all generic $n$-gons equal to one another?} 
These experiments suggest that the answer is positive for $n \leq 15$:  
in fact, in all cases we observe that either no exact NMF is found for the ten randomly generated matrices, or at least some are found for all of them. 
For $n = 16$, it is less clear whether this is true: we were only able to compute a rank-10 exact NMF for two of the generated matrices. 
This might be because these matrices are not fully generic, or because, for $n \geq 16$,  generic $n$-gons might have different extension complexities, or because our heuristic fails to compute the exact NMF of such instances. 
We leave the investigation of these issues for further research.  \\

The validity of conjecture~\ref{conjsm} would imply the following.
\begin{cor*}[of Conjecture \ref{conjsm}] 
The nonnegative rank of any rank-3 nonnegative matrix $X$ satisfies 
\[
\rank_+(X) \; \leq \; \left\lfloor \frac{\min(m,n)+6}{2} \right\rfloor . 
\]
\end{cor*} 
\begin{proof}[Sketch of proof] 
According to the geometric interpretation of NMF detailed in \cite{GG10b}, an exact NMF problem for a nonnegative rank-three matrix $X$ can be equivalently seen as a nested polygon problem: the matrix $X$ corresponds to the slack of the $n$ vertices of a inner polygon with respect to the $m$ edges of a outer polygon, and the goal is not find a polytope (potentially of higher dimension) that lies in between the two given polygons.  
The worst-case scenario happens when the inner and outer polytopes coincide, that is, when the matrix $X$ corresponds to the slack matrix of the outer polygon (see Section~\ref{exactNMFds}).  
Hence, replacing the $n$ vertices of the inner polygon by the $m$ vertices of the outer polygon can only increase the nonnegative rank of the corresponding matrix (hence Conjecture~\ref{conjsm} would apply). 
Moreover, transposing the matrix $X$ amounts to taking the polar of the polygons which interchanges the roles of the inner and outer polygons.  Hence the nonnegative rank of a rank-3 nonnegative $X$ is smaller than the largest nonnegative rank among slack matrices of $\min(m,n)$-gons. 
We refer the reader to \cite{Shit14} for more details.  
\end{proof}

\subsection{Extension Complexity of the Correlation Polytope}

The convex hull of all $n \times n$ rank-one matrices is called the correlation polytope, 
and we denote ts slack matrix $COR(n)$.  
It was proved in \cite{FMPTdW12} that there exists a positive constant $C$ for which $\rank_+(COR(n))\geq 2^{Cn}$.
This result was improved in \cite{KW13} where it is shown that $\rank_+(COR(n))\geq {1.5}^{n}$. 

Let us define the following ${2^n \times 2^n}$ matrix, a special instance of UDISJ matrices (see Section~\ref{exactNMFds}), 
for which rows and columns are indexed by vectors $a, b \in \{0,1\}^n$ and such that 
\[
M_n(a,b) 
= \left(1-|a^Tb|\right)^2.
\]
The matrix $M_n$ is a submatrix of the slack matrix of the correlation polytope~\cite{FMPTdW12}. 
For $n=3,4,5,6$, Hybrid was not able to compute any exact NMF with $r=2^n-1$ after 1000 runs. 
This suggests the following conjecture. 
\begin{conj} \label{conjcorrelpol}
The submatrix $M_n$ of the slack matrix of the correlation polytope has full nonnegative rank, that is,
\[
\rank_+(M_n) = 2^n. 
\] 
This would imply that $\rank_+(COR(n)) \geq 2^{n}$. 
\end{conj} 
(Note that the rank of $M_n$ is equal to $\frac{n(n+1)}{2} + 1$ for $n \leq 11$. For higher $n$, the matrix is too large to fit in memory.)

\section{Conclusion and Further Research}

We have proposed two new heuristics along with a hybridization for exact nonnegative matrix factorization, and demonstrated that they outperform simpler multi-start strategies when benchmarked on a variety of nonnegative matrices relevant for applications. On the way we proposed a novel efficient initialization strategy, and observed that HALS and A-HALS were suitable as local NMF algorithms when performing exact NMF. \\

Future research includes the development of new and more efficient heuristics. 
Also, heuristics can be sensitive to their parameters, especially for matrices for which it is difficult to compute an exact NMF.  
Hence potential future work also includes fine-tuning the parameters depending on the problem at hand (size of the matrix, difficulty of the corresponding NMF problem, etc.). 

The heuristics presented here can readily be applied to find good local minimum for the approximate NMF problem (that is, to compute $WH \approx  X$), which is particularly useful for real-world applications such as document classification and hyperspectral unmixing. 
Therefore, it would be an interesting direction for further research to fine-tune and compare heuristics in this context. 

So far, we have tested our algorithms on a limited number of nonnegative matrices. 
It would be good in the future to have a larger library of nonnegative matrices at our disposal, in order to better understand the behavior of the heuristics. With that goal in mind, we will keep our library updated on~\url{https://sites.google.com/site/exactnmf} and welcome submission of nonnegative matrices, especially those for which computing an exact factorization is still a challenge. 

Finally, it is important to recall that, strictly speaking, factorizations presented in this paper were not exact, because they were only computed up to machine precision; see Remark~\ref{trullyexact}. It would therefore also be useful to develop some rounding strategies to transform a high accuracy solution (e.g., $10^{-16}$ precision) into an exact NMF, when possible. 
(This was for example done manually for the example of Section~\ref{krone} where $\rank_+(A) = 4$ and $\rank_+(A \otimes A)= 12$.)

\bibliographystyle{spmpsci}
\bibliography{Biography}

\appendix

\section{Sensitivity to the Parameters $\alpha$ and $\Delta t$}  \label{appA} 

In this section, we show some numerical results to stress out that the heuristics are not too sensitive (in terms of number of exact NMF's found) to the parameters $\alpha$ and $\Delta t$ of the local search heuristic (Algorithm~\ref{alg:iterations}), 
as long as they are chosen sufficiently large; see Tables~\ref{alphadetlat} and \ref{choixdetlat}. 
This is the reason why we selected the rather conservative values of $\alpha = 0.99$ and $\Delta t = 1$ in this paper.  
\begin{table}[h]
\footnotesize 
\begin{center}
\caption{Comparison of different values of $\alpha$ with $\Delta t=1$ combined with multi-start~2.}
\begin{tabular}{|c||c|c|c|c|}
\hline
 & $\alpha=0.9999$ & $\alpha=0.99$ & $\alpha=0.9$ & $\alpha=0.5$ \\
\hline
LEDM6 & \textbf{8/10} (2.6) & 7/10 (3.2) & \textbf{8/10} (2.8) & \textbf{8/10} (\underline{2.1}) \\
LEDM8 & \textbf{5/30} (28.8) & \textbf{6/40} (20) & \textbf{5/30} (\underline{17.1}) & \textbf{5/30} (\underline{16.1}) \\
LEDM12 & 0/100 ($\sim$) & 0/100 ($\sim$) & 0/100 ($\sim$) & 0/100 ($\sim$) \\
LEDM16 & 0/100 ($\sim$) & 0/100 ($\sim$) & 0/100 ($\sim$) & 0/100 ($\sim$) \\
LEDM32 & 0/100 ($\sim$) & 0/100 ($\sim$) & 0/100 ($\sim$) & 0/100 ($\sim$) \\
6-G & \textbf{10/10} (\underline{1.7}) & \textbf{10/10} (\underline{1.8}) & \textbf{10/10} (2.1) & \textbf{10/10} (2) \\
7-G & \textbf{10/10} (\underline{1.8}) & \textbf{10/10} (2.2) & \textbf{10/10} (2.2) & \textbf{10/10} (2.1) \\
8-G & \textbf{9/10} (\underline{2.4}) & \textbf{9/10} (\underline{2.3}) & \textbf{9/10} (\underline{2.3}) & \textbf{9/10} (\underline{2.3}) \\
9-G & 5/10 (4.7) & \textbf{6/10} (\underline{4.2}) & 5/10 (4.8) & 5/10 (4.7) \\
16-G & 0/100 ($\sim$) & 0/100 ($\sim$) & 0/100 ($\sim$) & 0/100 ($\sim$) \\
32-G & 0/100 ($\sim$) & 0/100 ($\sim$) & 0/100 ($\sim$) & 0/100 ($\sim$) \\
20-D & 0/100 ($\sim$) & 0/100 ($\sim$) & 0/100 ($\sim$) & 0/100 ($\sim$) \\
24-C & 0/100 ($\sim$) & 0/100 ($\sim$) & 0/100 ($\sim$) & 0/100 ($\sim$) \\
UDISJ4 & \textbf{10/10} (2.4) & \textbf{10/10} (\underline{2.1}) & \textbf{10/10} (2.4) & \textbf{10/10} (2.4) \\
UDISJ5 & \textbf{6/20} (\underline{23.1}) & 5/30 (36.4) & 6/40 (40.6) & 3/100 (179.3) \\
UDISJ6 & 0/100 ($\sim$) & 0/100 ($\sim$) & 0/100 ($\sim$) & 0/100 ($\sim$) \\
RND1 & \textbf{10/10} (\underline{2.2}) & \textbf{10/10} (\underline{2.3}) & \textbf{10/10} (2.6) & \textbf{10/10} (2.5) \\
RND3 & \textbf{10/10} (2.7) & \textbf{10/10} (\underline{2.4}) & \textbf{10/10} (\underline{2.2}) & \textbf{10/10} (2.7) \\
\hline
\end{tabular}
\label{alphadetlat}
\end{center}
\end{table} 
\begin{table}[h]
\footnotesize 
\begin{center}
\caption{Comparison of different values of $\Delta t$ with $\alpha = 0.99$ combined with multi-start~2.}
\begin{tabular}{|c||c|c|c|c|c|c|}
\hline
 & $\Delta t=0.001$ & $\Delta t=0.01$ & $\Delta t=0.05$ & $\Delta t=0.1$ & $\Delta t=1$ & $\Delta t=2$ \\
\hline
LEDM6 & \textbf{8/10} (\underline{1.5}) & 7/10 (1.7) & \textbf{8/10} (\underline{1.5}) & \textbf{8/10} (\underline{1.5}) & 7/10 (3.2) & \textbf{8/10} (4.4) \\
LEDM8 & 5/50 (12.6) & 5/50 (13) & \textbf{5/30} (\underline{8.6}) & \textbf{5/30} (9.8) & \textbf{6/40} (20) & \textbf{5/30} (33.4) \\
LEDM12 & 0/100 ($\sim$) & 0/100 ($\sim$) & 0/100 ($\sim$) & 0/100 ($\sim$) & 0/100 ($\sim$) & 0/100 ($\sim$) \\
LEDM16 & 0/100 ($\sim$) & 0/100 ($\sim$) & 0/100 ($\sim$) & 0/100 ($\sim$) & 0/100 ($\sim$) & 0/100 ($\sim$) \\
LEDM32 & 0/100 ($\sim$) & 0/100 ($\sim$) & 0/100 ($\sim$) & 0/100 ($\sim$) & 0/100 ($\sim$) & 0/100 ($\sim$) \\
6-G & \textbf{10/10} (\underline{1.2}) & \textbf{10/10} (\underline{1.1}) & \textbf{10/10} (\underline{1.1}) & \textbf{10/10} (\underline{1.2}) & \textbf{10/10} (1.8) & \textbf{10/10} (3.1) \\
7-G & \textbf{10/10} (\underline{1.2}) & \textbf{10/10} (\underline{1.2}) & \textbf{10/10} (\underline{1.2}) & \textbf{10/10} (\underline{1.3}) & \textbf{10/10} (2.2) & \textbf{10/10} (3.1) \\
8-G & \textbf{9/10} (\underline{1.4}) & \textbf{9/10} (\underline{1.4}) & \textbf{9/10} (\underline{1.4}) & \textbf{9/10} (\underline{1.4}) & \textbf{9/10} (2.3) & \textbf{9/10} (3.5) \\
9-G & \textbf{6/10} (\underline{2.3}) & 5/10 (2.7) & 5/10 (2.7) & 5/10 (2.8) & \textbf{6/10} (4.2) & 5/10 (8.4) \\
16-G & 0/100 ($\sim$) & 0/100 ($\sim$) & 0/100 ($\sim$) & 0/100 ($\sim$) & 0/100 ($\sim$) & 0/100 ($\sim$) \\
32-G & 0/100 ($\sim$) & 0/100 ($\sim$) & 0/100 ($\sim$) & 0/100 ($\sim$) & 0/100 ($\sim$) & 0/100 ($\sim$) \\
20-D & 0/100 ($\sim$) & 0/100 ($\sim$) & 0/100 ($\sim$) & 0/100 ($\sim$) & 0/100 ($\sim$) & 0/100 ($\sim$) \\
24-C & 0/100 ($\sim$) & 0/100 ($\sim$) & 0/100 ($\sim$) & 0/100 ($\sim$) & 0/100 ($\sim$) & 0/100 ($\sim$) \\
UDISJ4 & \textbf{10/10} (\underline{1.5}) & \textbf{10/10} (\underline{1.5}) & \textbf{10/10} (\underline{1.5}) & \textbf{10/10} (\underline{1.6}) & \textbf{10/10} (2.1) & \textbf{10/10} (3.4) \\
UDISJ5 & 1/100 (606.2) & 1/100 (614.9) & 4/100 (153.6) & 2/100 (306.3) & 5/30 (\underline{36.4}) & \textbf{5/20} (\underline{37.5}) \\
UDISJ6 & 0/100 ($\sim$) & 0/100 ($\sim$) & 0/100 ($\sim$) & 0/100 ($\sim$) & 0/100 ($\sim$) & 0/100 ($\sim$) \\
RND1 & \textbf{10/10} (\underline{1.9}) & \textbf{10/10} (\underline{1.8}) & \textbf{10/10} (\underline{1.9}) & \textbf{10/10} (\underline{1.9}) & \textbf{10/10} (2.3) & \textbf{10/10} (3.8) \\
RND3 & \textbf{10/10} (\underline{1.9}) & \textbf{10/10} (\underline{1.8}) & \textbf{10/10} (\underline{1.9}) & \textbf{10/10} (\underline{1.9}) & \textbf{10/10} (2.4) & \textbf{10/10} (3.8) \\
\hline
\end{tabular}
\label{choixdetlat}
\end{center}
\end{table}

In practice however, it would be good to start the heuristics with smaller values for $\alpha$ and $\Delta t$ and 
increase them progressively if the heuristic fails to identify exact NMF's: for easily factorizable matrices (such as the randomly generated ones) it does not make sense to choose large parameters, while for difficult matrices choosing $\alpha$ and $\Delta t$ too small does not allow the heuristics to find exact NMF's because convergence of NMF algorithms can, in some cases, be too slow.  \\

\section{Parameters for Simulated Annealing}  \label{appB}

Table~\ref{tableinitrec} shows the performance of SA for different initialization strategies described in Section~\ref{rdninit}  
(for $T_0 = 0.1$, $T_{end} = 10^{-4}$, $J=2$, $N = 100$ and $K = 50$): it appears that SPARSE10 works on average the best hence we keep this initialization for SA. In particular, it is interesting to notice that SPARSE10 is able to compute exact NMF's of 32-G while the other initializations have much more difficulties (only SPARSE00 finds one exact NMF). 
\begin{table}[h]
\footnotesize
\begin{center}
\caption{Comparison of the different initialization strategies combined with SA.}
\begin{tabular}{|c||c|c|c|c|c|}
\hline
 & sparse 00 & sparse 10 & sparse 01 & sparse 11 & rndcube \\
\hline
LEDM6 & \textbf{10/10} (19.5) & \textbf{10/10} (\underline{17}) & \textbf{10/10} (20.4) & \textbf{10/10} (\underline{16.5}) & \textbf{10/10} (19.4) \\
LEDM8 & \textbf{10/10} (57.9) & \textbf{10/10} (\underline{44.8}) & \textbf{10/10} (59) & \textbf{9/10} (49.9) & \textbf{10/10} (63.3) \\
LEDM12 & \textbf{9/10} (\underline{26.5}) & 6/10 (30) & 11/20 (45.1) & 8/10 (\underline{28}) & 10/20 (51) \\
LEDM16 & 7/20 (125.9) & \textbf{11/20} (\underline{65.6}) & 5/10 (99.2) & 6/20 (112.8) & 6/20 (132.6) \\
LEDM32 & \textbf{5/90} (\underline{711.7}) & \textbf{3/100} (1016.9) & \textbf{1/100} (3728.4) & \textbf{2/100} (1447.5) & 0/100 ($\sim$) \\
6-G & \textbf{10/10} (\underline{1.2}) & \textbf{10/10} (\underline{1.2}) & \textbf{10/10} (\underline{1.1}) & \textbf{10/10} (1.3) & \textbf{10/10} (\underline{1.2}) \\
7-G & \textbf{10/10} (\underline{3.5}) & \textbf{10/10} (\underline{3.5}) & \textbf{9/10} (72.5) & \textbf{10/10} (\underline{3.4}) & \textbf{10/10} (\underline{3.5}) \\
8-G & \textbf{10/10} (17.2) & \textbf{10/10} (\underline{13.8}) & \textbf{10/10} (19.6) & \textbf{10/10} (15.9) & \textbf{10/10} (16.1) \\
9-G & \textbf{10/10} (22.7) & \textbf{10/10} (21.3) & \textbf{10/10} (23) & \textbf{10/10} (\underline{18}) & \textbf{10/10} (24.1) \\
16-G & 6/20 (87.6) & \textbf{8/20} (\underline{61.4}) & 7/40 (150.4) & 5/60 (287.7) & 6/40 (182.7) \\
32-G & 0/100 ($\sim$) & \textbf{3/100} (\underline{999.5}) & 0/100 ($\sim$) & 0/100 ($\sim$) & 0/100 ($\sim$) \\
20-D & \textbf{10/10} (7.5) & \textbf{10/10} (\underline{4.2}) & \textbf{10/10} (8.1) & \textbf{10/10} (10) & \textbf{10/10} (8) \\
24-C & \textbf{10/10} (4.4) & \textbf{10/10} (3.6) & \textbf{10/10} (\underline{3.1}) & \textbf{10/10} (4.5) & \textbf{10/10} (3.7) \\
UDISJ4 & \textbf{10/10} (\underline{1.2}) & \textbf{10/10} (\underline{1.2}) & \textbf{10/10} (\underline{1.1}) & \textbf{10/10} (\underline{1.1}) & \textbf{10/10} (\underline{1.1}) \\
UDISJ5 & \textbf{10/10} (2.9) & \textbf{10/10} (\underline{2.3}) & \textbf{10/10} (3) & \textbf{10/10} (2.7) & \textbf{10/10} (3.8) \\
UDISJ6 & \textbf{10/10} (\underline{8.3}) & \textbf{10/10} (\underline{8.1}) & \textbf{10/10} (52.4) & \textbf{10/10} (13.5) & \textbf{10/10} (43.9) \\
RND1 & \textbf{10/10} (\underline{1.1}) & \textbf{10/10} (\underline{1.1}) & \textbf{10/10} (\underline{1.1}) & \textbf{10/10} (\underline{1.1}) & \textbf{10/10} (\underline{1.1}) \\
RND3 & \textbf{10/10} (\underline{1.1}) & \textbf{10/10} (\underline{1.1}) & \textbf{10/10} (\underline{1.1}) & \textbf{10/10} (\underline{1.1}) & \textbf{10/10} (\underline{1.1}) \\
\hline
\end{tabular}
\label{tableinitrec}
\end{center}
\end{table} 

Table~\ref{recuitParamTend} shows the performance for different values of $T_{end}$ (for $J=2$, $N = 100$ and $K = 50$): it appears that the value $T_{end} = 10^{-4}$ for the final temperature works well. 
\begin{table}[h]
\footnotesize 
\begin{center}
\caption{Performance of Simulated Annealing for different values of $T_{end}$ ($J=2$, $N = 100$ and $K = 50$).}
\begin{tabular}{|c||c|c|c|c|c|}
\hline
 & $T_{end}=10^{-2}$ & $T_{end}=10^{-3}$ & $T_{end}=10^{-4}$ & $T_{end}=10^{-5}$ & $T_{end}=10^{-6}$ \\
\hline
LEDM6 & \textbf{10/10} (17.8) & \textbf{10/10} (19) & \textbf{10/10} (17) & \textbf{10/10} (\underline{12.3}) & \textbf{10/10} (13.6) \\
LEDM8 & \textbf{10/10} (54.5) & \textbf{10/10} (57) & \textbf{10/10} (\underline{44.8}) & \textbf{10/10} (62.3) & \textbf{9/10} (\underline{49.1}) \\
LEDM12 & 6/60 (225.4) & 6/20 (63.9) & 6/10 (\underline{30}) & 5/10 (33) & \textbf{7/10} (\underline{29.4}) \\
LEDM16 & 5/100 (488.7) & 5/50 (223) & \textbf{11/20} (\underline{65.6}) & 5/10 (84.1) & 6/20 (100.1) \\
LEDM32 & 0/100 ($\sim$) & 0/100 ($\sim$) & \textbf{3/100} (\underline{1016.9}) & 0/100 ($\sim$) & \textbf{2/100} (1561) \\
6-G & \textbf{10/10} (\underline{1.2}) & \textbf{10/10} (\underline{1.2}) & \textbf{10/10} (\underline{1.2}) & \textbf{10/10} (\underline{1.2}) & \textbf{10/10} (\underline{1.2}) \\
7-G & \textbf{10/10} (3.9) & \textbf{10/10} (\underline{3.8}) & \textbf{10/10} (\underline{3.5}) & \textbf{10/10} (3.9) & \textbf{10/10} (4.2) \\
8-G & \textbf{10/10} (16.5) & \textbf{10/10} (17.4) & \textbf{10/10} (13.8) & \textbf{10/10} (\underline{11.4}) & \textbf{10/10} (13.4) \\
9-G & \textbf{10/10} (21) & \textbf{10/10} (17.2) & \textbf{10/10} (21.3) & \textbf{10/10} (\underline{12.7}) & \textbf{10/10} (15.8) \\
16-G & 5/80 (352.4) & 6/20 (66.3) & 8/20 (61.4) & 6/30 (97.4) & \textbf{6/10} (\underline{23.7}) \\
32-G & 0/100 ($\sim$) & 0/100 ($\sim$) & \textbf{3/100} (\underline{999.5}) & \textbf{3/100} (\underline{931.3}) & \textbf{2/100} (1331.3) \\
20-D & \textbf{10/10} (\underline{4.1}) & \textbf{10/10} (5.9) & \textbf{10/10} (\underline{4.2}) & \textbf{10/10} (7.6) & \textbf{10/10} (5.2) \\
24-C & \textbf{10/10} (3.9) & \textbf{10/10} (\underline{2.1}) & \textbf{10/10} (3.6) & \textbf{10/10} (2.9) & \textbf{10/10} (2.7) \\
UDISJ4 & \textbf{10/10} (\underline{1.2}) & \textbf{10/10} (\underline{1.2}) & \textbf{10/10} (\underline{1.2}) & \textbf{10/10} (\underline{1.1}) & \textbf{10/10} (\underline{1.2}) \\
UDISJ5 & \textbf{10/10} (\underline{2.3}) & \textbf{10/10} (\underline{2.4}) & \textbf{10/10} (\underline{2.3}) & \textbf{10/10} (\underline{2.5}) & \textbf{10/10} (\underline{2.3}) \\
UDISJ6 & \textbf{10/10} (9) & \textbf{10/10} (9.4) & \textbf{10/10} (\underline{8.1}) & \textbf{10/10} (\underline{7.7}) & \textbf{10/10} (8.6) \\
RND1 & \textbf{10/10} (\underline{1.1}) & \textbf{10/10} (\underline{1.1}) & \textbf{10/10} (\underline{1.1}) & \textbf{10/10} (\underline{1.1}) & \textbf{10/10} (\underline{1.1}) \\
RND3 & \textbf{10/10} (\underline{1.1}) & \textbf{10/10} (\underline{1.1}) & \textbf{10/10} (\underline{1.1}) & \textbf{10/10} (\underline{1.1}) & \textbf{10/10} (\underline{1.1}) \\
\hline
\end{tabular}
\label{recuitParamTend}
\end{center}
\end{table}  

Table~\ref{recuitParamKN} shows the performance for different values of $N$ and $K$, for $T_{end} = 10^{-4}$ and $J = 2$. 
It seems that $K = 50$ and $N = 100$ is a good compromise between number of exact NMF's found and computational time. 
\begin{sidewaystable}[h]
\tiny
\begin{center}
\caption{Performance of Simulated Annealing for different values of $K$ and $N$ ($T_{end} = 10^{-4}$ and $J = 2$).}
\begin{tabular}{|c||c|c|c|c|c|c|c|c|c|c|c|c|}
\hline
 & \multicolumn{3}{c|}{$K=10$} & \multicolumn{3}{c|}{$K=20$} & \multicolumn{3}{c|}{$K=50$} & \multicolumn{3}{c|}{$K=100$} \\
\cline{2-13}
 & $N=10$ & $N=50$ & $N=100$ & $N=10$ & $N=50$ & $N=100$ & $N=10$ & $N=50$ & $N=100$ & $N=10$ & $N=50$ & $N=100$ \\
\hline
LEDM6 & 6/10 (3.1) & \textbf{10/10} (3.1) & \textbf{10/10} (4.9) & 8/10 (\underline{2.8}) & \textbf{10/10} (5) & \textbf{10/10} (8.7) & 8/10 (4.6) & \textbf{10/10} (11.1) & \textbf{10/10} (17) & \textbf{10/10} (5.9) & \textbf{10/10} (19) & \textbf{10/10} (38.4) \\
LEDM8 & 7/20 (53.1) & 7/10 (\underline{45}) & 8/10 (\underline{48}) & \textbf{9/10} (49) & \textbf{10/10} (\underline{48}) & \textbf{10/10} (49.8) & \textbf{9/10} (\underline{44}) & \textbf{9/10} (59.4) & \textbf{10/10} (\underline{44.8}) & \textbf{9/10} (49.4) & \textbf{10/10} (62.1) & \textbf{10/10} (71.5) \\
LEDM12 & 5/80 (47.2) & 7/20 (\underline{13.2}) & 5/30 (40.8) & 5/30 (21.6) & 5/10 (\underline{12.7}) & 6/10 (18.4) & 6/20 (15.4) & 6/10 (23) & 6/10 (30) & 5/20 (31.5) & 7/10 (33.8) & \textbf{9/10} (55.9) \\
LEDM16 & 5/70 (98.1) & 5/50 (115.4) & 5/80 (184.6) & 5/30 (76.7) & 6/30 (93.9) & 5/20 (91.2) & 7/60 (103.5) & 5/20 (100.9) & 11/20 (\underline{65.6}) & 6/40 (116.4) & \textbf{7/10} (81.2) & 6/10 (156) \\
LEDM32 & 0/100 ($\sim$) & 0/100 ($\sim$) & 0/100 ($\sim$) & 0/100 ($\sim$) & 0/100 ($\sim$) & 0/100 ($\sim$) & 0/100 ($\sim$) & \textbf{3/100} (706.4) & \textbf{3/100} (1016.9) & \textbf{1/100} (1318) & \textbf{5/70} (\underline{571.8}) & \textbf{5/70} (1049) \\
6-G & \textbf{10/10} (1.6) & \textbf{10/10} (2.5) & \textbf{10/10} (\underline{1.2}) & \textbf{10/10} (2) & \textbf{10/10} (3.4) & \textbf{10/10} (\underline{1.2}) & \textbf{10/10} (3.5) & \textbf{10/10} (2.7) & \textbf{10/10} (\underline{1.2}) & \textbf{10/10} (5.4) & \textbf{10/10} (3.9) & \textbf{10/10} (\underline{1.2}) \\
7-G & 7/10 (2.7) & \textbf{10/10} (3) & \textbf{10/10} (\underline{2.4}) & \textbf{9/10} (\underline{2.4}) & \textbf{10/10} (4.7) & \textbf{10/10} (3.8) & \textbf{10/10} (3.7) & \textbf{10/10} (9.3) & \textbf{10/10} (3.5) & \textbf{10/10} (5.9) & \textbf{10/10} (18.2) & \textbf{10/10} (4.5) \\
8-G & \textbf{10/10} (\underline{1.6}) & \textbf{10/10} (3.1) & \textbf{10/10} (4.8) & \textbf{10/10} (2.1) & \textbf{10/10} (5) & \textbf{10/10} (7.3) & \textbf{10/10} (3.8) & \textbf{10/10} (11) & \textbf{10/10} (13.8) & \textbf{10/10} (5.9) & \textbf{10/10} (21.1) & \textbf{10/10} (28) \\
9-G & 8/10 (\underline{2.3}) & 7/10 (5.1) & \textbf{10/10} (5.5) & 8/10 (3) & \textbf{10/10} (5.5) & \textbf{10/10} (9.9) & \textbf{9/10} (4.4) & \textbf{10/10} (11.8) & \textbf{10/10} (21.3) & \textbf{9/10} (7.1) & \textbf{10/10} (23.3) & \textbf{10/10} (42.4) \\
16-G & 5/30 (\underline{15.1}) & 7/60 (38.8) & 5/40 (56.5) & 5/30 (18.7) & 6/20 (21.7) & 5/20 (46) & 6/20 (16.3) & 9/20 (31.5) & 8/20 (61.4) & 10/20 (\underline{14.4}) & 9/20 (58.7) & \textbf{8/10} (63.6) \\
32-G & 1/100 (410.7) & 0/100 ($\sim$) & 1/100 (895.7) & 0/100 ($\sim$) & 2/100 (458.4) & 2/100 (740) & 2/100 (\underline{341.1}) & 3/100 (606.6) & 3/100 (999.5) & 3/100 (\underline{330.4}) & 5/60 (405.1) & \textbf{5/40} (522.4) \\
20-D & 6/10 (\underline{3.5}) & 8/10 (4.8) & \textbf{9/10} (4.9) & 5/10 (5.7) & 7/10 (9.3) & \textbf{9/10} (6.4) & \textbf{10/10} (4.4) & \textbf{10/10} (14.3) & \textbf{10/10} (4.2) & \textbf{9/10} (8.4) & \textbf{10/10} (27.7) & \textbf{10/10} (11.8) \\
24-C & 8/10 (\underline{2.5}) & \textbf{9/10} (4.9) & \textbf{10/10} (2.9) & 8/10 (3.3) & 8/10 (9.6) & \textbf{9/10} (4.4) & \textbf{10/10} (5) & \textbf{10/10} (17.1) & \textbf{10/10} (3.6) & \textbf{10/10} (8.7) & \textbf{10/10} (30.5) & \textbf{10/10} (4.2) \\
UDISJ4 & \textbf{10/10} (1.7) & \textbf{10/10} (\underline{1.3}) & \textbf{10/10} (\underline{1.2}) & \textbf{10/10} (2.4) & \textbf{10/10} (1.4) & \textbf{10/10} (\underline{1.2}) & \textbf{10/10} (4.4) & \textbf{10/10} (\underline{1.3}) & \textbf{10/10} (\underline{1.2}) & \textbf{10/10} (7.4) & \textbf{10/10} (1.4) & \textbf{10/10} (\underline{1.2}) \\
UDISJ5 & \textbf{10/10} (3.9) & \textbf{10/10} (13.7) & \textbf{10/10} (2.7) & \textbf{10/10} (6.8) & \textbf{10/10} (26.3) & \textbf{10/10} (3.5) & \textbf{10/10} (15.5) & \textbf{10/10} (55.6) & \textbf{10/10} (\underline{2.3}) & \textbf{10/10} (28.6) & \textbf{10/10} (98.7) & \textbf{10/10} (3.5) \\
UDISJ6 & \textbf{9/10} (\underline{6.9}) & \textbf{10/10} (22.9) & \textbf{10/10} (\underline{7.1}) & \textbf{10/10} (11.4) & \textbf{10/10} (42.3) & \textbf{10/10} (8.8) & \textbf{10/10} (27.4) & \textbf{10/10} (63.5) & \textbf{10/10} (8.1) & \textbf{10/10} (48.5) & \textbf{10/10} (131.5) & \textbf{10/10} (11.3) \\
RND1 & \textbf{10/10} (1.9) & \textbf{10/10} (1.3) & \textbf{10/10} (\underline{1.1}) & \textbf{10/10} (2.7) & \textbf{10/10} (1.3) & \textbf{10/10} (\underline{1.1}) & \textbf{10/10} (5.1) & \textbf{10/10} (1.3) & \textbf{10/10} (\underline{1.1}) & \textbf{10/10} (8.5) & \textbf{10/10} (1.3) & \textbf{10/10} (\underline{1.1}) \\
RND3 & \textbf{10/10} (1.9) & \textbf{10/10} (\underline{1.1}) & \textbf{10/10} (\underline{1.1}) & \textbf{10/10} (2.7) & \textbf{10/10} (\underline{1.1}) & \textbf{10/10} (\underline{1.1}) & \textbf{10/10} (5.1) & \textbf{10/10} (\underline{1.1}) & \textbf{10/10} (\underline{1.1}) & \textbf{10/10} (8.5) & \textbf{10/10} (\underline{1.1}) & \textbf{10/10} (\underline{1.1}) \\
\hline
\end{tabular}
\label{recuitParamKN}
\end{center}
\end{sidewaystable} 

Table~\ref{recuitParamJ} shows the performance for different values of $J$ (for $T_{end} = 10^{-4}$, $K = 50$ and $N = 100$), and shows that $J = 2$ performs the best. 
\begin{table}[h]
\footnotesize 
\begin{center}
\caption{Performance of Simulated Annealing for different values of $J$ ($T_{end} = 10^{-4}$, $K = 50$ and $N = 100$).}
\begin{tabular}{|c||c|c|c|c|}
\hline
 & $|J|=1$ & $|J|=2$ & $|J|=3$ & $|J|=4$ \\
\hline
LEDM6 & \textbf{10/10} (20.5) & \textbf{10/10} (\underline{17}) & \textbf{10/10} (20.3) & \textbf{10/10} (19.4) \\
LEDM8 & \textbf{10/10} (54.4) & \textbf{10/10} (\underline{44.8}) & \textbf{10/10} (64.4) & \textbf{10/10} (61.3) \\
LEDM12 & \textbf{10/10} (\underline{25}) & 6/10 (30) & 5/10 (50.3) & 7/20 (67.7) \\
LEDM16 & 5/10 (100.9) & \textbf{11/20} (\underline{65.6}) & 7/20 (115.1) & 6/20 (99.4) \\
LEDM32 & \textbf{1/100} (3655.9) & \textbf{3/100} (\underline{1016.9}) & \textbf{1/100} (3688) & \textbf{2/100} (1798) \\
6-G & \textbf{10/10} (\underline{1.2}) & \textbf{10/10} (\underline{1.2}) & \textbf{10/10} (\underline{1.2}) & \textbf{10/10} (\underline{1.2}) \\
7-G & \textbf{10/10} (\underline{2.2}) & \textbf{10/10} (3.5) & \textbf{10/10} (5.1) & \textbf{10/10} (10.4) \\
8-G & \textbf{10/10} (17.5) & \textbf{10/10} (\underline{13.8}) & \textbf{10/10} (16.5) & \textbf{10/10} (20.4) \\
9-G & \textbf{10/10} (\underline{19.5}) & \textbf{10/10} (\underline{21.3}) & \textbf{10/10} (22.8) & \textbf{10/10} (23.5) \\
16-G & \textbf{6/10} (\underline{44.5}) & 8/20 (61.4) & 5/20 (108.2) & 6/60 (268.2) \\
32-G & \textbf{5/90} (\underline{613.1}) & \textbf{3/100} (999.5) & 0/100 ($\sim$) & \textbf{1/100} (3377) \\
20-D & \textbf{9/10} (8) & \textbf{10/10} (\underline{4.2}) & \textbf{10/10} (8.6) & \textbf{10/10} (19.8) \\
24-C & \textbf{10/10} (4) & \textbf{10/10} (\underline{3.6}) & \textbf{10/10} (4.8) & \textbf{10/10} (6.5) \\
UDISJ4 & \textbf{10/10} (\underline{1.3}) & \textbf{10/10} (\underline{1.2}) & \textbf{10/10} (\underline{1.2}) & \textbf{10/10} (\underline{1.2}) \\
UDISJ5 & \textbf{10/10} (3.7) & \textbf{10/10} (\underline{2.3}) & \textbf{10/10} (2.8) & \textbf{10/10} (3) \\
UDISJ6 & \textbf{10/10} (11) & \textbf{10/10} (8.1) & \textbf{10/10} (\underline{6.8}) & \textbf{10/10} (8.1) \\
RND1 & \textbf{10/10} (\underline{1.1}) & \textbf{10/10} (\underline{1.1}) & \textbf{10/10} (\underline{1.2}) & \textbf{10/10} (\underline{1.1}) \\
RND3 & \textbf{10/10} (\underline{1.1}) & \textbf{10/10} (\underline{1.1}) & \textbf{10/10} (\underline{1.1}) & \textbf{10/10} (\underline{1.1}) \\
\hline
\end{tabular}
\label{recuitParamJ}
\end{center}
\end{table}

\section{Parameters for the Rank-by-Rank Heuristic}  \label{appC}

Table~\ref{tableinitrbr} shows the performance of RBR for the different initialization strategies   
(for $N = 100$ and $K = 50$): SPARSE10 works on average the best. As for SA, it allows to compute exact NMF's of 32-G (6/10) while all other initializations fail. 
\begin{table}[h]
\footnotesize
\begin{center}
\caption{Comparison of the different initialization strategies combined with RBR.}
\begin{tabular}{|c||c|c|c|c|c|}
\hline
 & sparse 00 & sparse 10 & sparse 01 & sparse 11 & rndcube \\
\hline
LEDM6 & \textbf{10/10} (\underline{1.4}) & \textbf{10/10} (\underline{1.4}) & \textbf{10/10} (\underline{1.4}) & \textbf{10/10} (\underline{1.4}) & \textbf{10/10} (\underline{1.4}) \\
LEDM8 & \textbf{10/10} (14.6) & \textbf{10/10} (15.8) & \textbf{10/10} (\underline{12.3}) & \textbf{10/10} (20.3) & \textbf{10/10} (\underline{11.7}) \\
LEDM12 & 5/30 (14.8) & \textbf{7/30} (10.1) & \textbf{6/20} (\underline{7.8}) & \textbf{7/30} (10.3) & \textbf{5/20} (9.4) \\
LEDM16 & 5/50 (40.4) & 5/30 (29.5) & 6/40 (31.5) & 5/30 (29.5) & \textbf{6/10} (\underline{17.9}) \\
LEDM32 & 0/100 ($\sim$) & 0/100 ($\sim$) & 0/100 ($\sim$) & 0/100 ($\sim$) & 0/100 ($\sim$) \\
6-G & \textbf{10/10} (\underline{1.4}) & \textbf{10/10} (\underline{1.4}) & \textbf{10/10} (\underline{1.4}) & \textbf{10/10} (\underline{1.4}) & \textbf{10/10} (\underline{1.4}) \\
7-G & \textbf{10/10} (\underline{1.5}) & \textbf{10/10} (\underline{1.5}) & \textbf{10/10} (\underline{1.5}) & \textbf{10/10} (\underline{1.5}) & \textbf{10/10} (\underline{1.5}) \\
8-G & \textbf{10/10} (\underline{1.5}) & \textbf{10/10} (\underline{1.5}) & \textbf{10/10} (\underline{1.5}) & \textbf{10/10} (\underline{1.5}) & \textbf{10/10} (\underline{1.5}) \\
9-G & \textbf{10/10} (\underline{1.6}) & \textbf{10/10} (\underline{1.6}) & \textbf{10/10} (\underline{1.6}) & \textbf{10/10} (\underline{1.6}) & \textbf{10/10} (\underline{1.6}) \\
16-G & \textbf{10/10} (\underline{1.8}) & \textbf{10/10} (\underline{1.8}) & \textbf{9/10} (2.1) & 5/10 (4.3) & 6/20 (8.1) \\
32-G & \textbf{7/20} (8.4) & \textbf{8/20} (\underline{7.1}) & 5/40 (26.7) & \textbf{7/20} (9.4) & 1/100 (421.1) \\
20-D & 8/10 (2.5) & \textbf{9/10} (2.1) & \textbf{10/10} (\underline{1.9}) & \textbf{9/10} (2.2) & 8/10 (2.3) \\
24-C & 7/10 (3.4) & \textbf{8/10} (\underline{2.9}) & \textbf{8/10} (\underline{2.9}) & \textbf{8/10} (\underline{3}) & 12/20 (4.1) \\
UDISJ4 & \textbf{10/10} (\underline{1.9}) & \textbf{10/10} (\underline{1.9}) & \textbf{10/10} (\underline{1.9}) & \textbf{10/10} (\underline{1.9}) & 8/10 (2.3) \\
UDISJ5 & \textbf{10/10} (\underline{5}) & \textbf{10/10} (\underline{4.8}) & \textbf{10/10} (\underline{4.9}) & \textbf{10/10} (\underline{4.9}) & \textbf{9/10} (5.5) \\
UDISJ6 & 6/20 (57.5) & 6/40 (116.8) & 7/10 (23.7) & \textbf{8/10} (\underline{21}) & 5/30 (106) \\
RND1 & \textbf{10/10} (\underline{2.2}) & \textbf{10/10} (\underline{2.2}) & \textbf{10/10} (\underline{2.2}) & \textbf{10/10} (\underline{2.2}) & \textbf{10/10} (\underline{2.2}) \\
RND3 & \textbf{10/10} (\underline{2.2}) & \textbf{10/10} (\underline{2.2}) & \textbf{10/10} (\underline{2.2}) & \textbf{10/10} (\underline{2.2}) & \textbf{10/10} (\underline{2.2}) \\
\hline
\end{tabular}
\label{tableinitrbr}
\end{center}
\end{table}

Table~\ref{rbrParamKN} gives the results for several values of the parameters $K$ and $N$. It is interesting to observe that when $K$ gets larger, the heuristic performs rather poorly in some cases (e.g., for the UDISJ6 matrix). The reason is that when $K$ increases, the heuristic tends to generate similar solutions: the ones obtained with Algorithm~\ref{alg:grpo} initialized with the best solution that can be obtained by combining the rank-$(k-1)$ solution with a rank-one one. In other words, the search domain that can be explored by RBR is reduced when $K$ increases. 

\begin{sidewaystable}[h]
\tiny
\begin{center}
\caption{Performance of the Rank-by-Rank heuristic for different values of $K$ and $N$.}
\begin{tabular}{|c||c|c|c|c|c|c|c|c|c|c|c|c|}
\hline
 & \multicolumn{3}{c|}{$K=1$} & \multicolumn{3}{c|}{$K=10$} & \multicolumn{3}{c|}{$K=50$} & \multicolumn{3}{c|}{$K=100$} \\
\cline{2-13}
 & $N=10$ & $N=50$ & $N=100$ & $N=10$ & $N=50$ & $N=100$ & $N=10$ & $N=50$ & $N=100$ & $N=10$ & $N=50$ & $N=100$ \\
\hline
LEDM6 & 8/30 (6.7) & 7/10 (2) & 5/10 (3.2) & \textbf{10/10} (\underline{1.2}) & \textbf{10/10} (1.4) & \textbf{10/10} (1.7) & \textbf{10/10} (1.4) & \textbf{10/10} (2.7) & \textbf{10/10} (4.1) & \textbf{10/10} (1.9) & \textbf{10/10} (4.4) & \textbf{10/10} (7.7) \\
LEDM8 & 6/10 (19.4) & 8/20 (13.6) & 6/10 (\underline{9.8}) & \textbf{10/10} (38.5) & \textbf{10/10} (15.8) & \textbf{10/10} (13.1) & \textbf{10/10} (20.8) & \textbf{10/10} (16) & \textbf{10/10} (28.8) & \textbf{10/10} (19.2) & \textbf{10/10} (16.5) & \textbf{10/10} (34.7) \\
LEDM12 & 5/50 (27.8) & 5/20 (10.2) & 7/20 (6.2) & 5/80 (33.8) & 7/30 (10.1) & \textbf{10/10} (\underline{2.2}) & 6/50 (21.2) & 9/20 (10.1) & \textbf{10/10} (6.2) & 6/40 (22) & 6/10 (11.7) & \textbf{10/10} (12.3) \\
LEDM16 & 0/100 ($\sim$) & 5/60 (59.1) & 5/40 (37.2) & 5/100 (73.8) & 5/30 (\underline{29.5}) & 6/40 (36.4) & 5/50 (44.2) & 7/50 (53.2) & 5/50 (98.7) & \textbf{6/20} (\underline{27.1}) & 5/70 (139.3) & 6/70 (197.6) \\
LEDM32 & \textbf{1/100} (\underline{545.3}) & 0/100 ($\sim$) & 0/100 ($\sim$) & 0/100 ($\sim$) & 0/100 ($\sim$) & 0/100 ($\sim$) & 0/100 ($\sim$) & 0/100 ($\sim$) & 0/100 ($\sim$) & 0/100 ($\sim$) & 0/100 ($\sim$) & 0/100 ($\sim$) \\
6-G & 5/10 (3.1) & 5/10 (2.8) & 6/20 (4) & \textbf{10/10} (\underline{1.2}) & \textbf{10/10} (1.4) & \textbf{10/10} (1.7) & \textbf{10/10} (1.4) & \textbf{10/10} (2.7) & \textbf{10/10} (4.1) & \textbf{10/10} (1.9) & \textbf{10/10} (4.3) & \textbf{10/10} (7.6) \\
7-G & 12/20 (2.4) & 7/10 (2) & 6/10 (2.1) & \textbf{10/10} (\underline{1.2}) & \textbf{10/10} (1.5) & \textbf{10/10} (1.9) & \textbf{10/10} (1.6) & \textbf{10/10} (3.2) & \textbf{10/10} (5) & \textbf{10/10} (2.1) & \textbf{10/10} (5.4) & \textbf{10/10} (9.6) \\
8-G & \textbf{10/10} (\underline{1.1}) & \textbf{10/10} (\underline{1.1}) & \textbf{10/10} (\underline{1.2}) & \textbf{10/10} (\underline{1.2}) & \textbf{10/10} (1.5) & \textbf{10/10} (1.9) & \textbf{10/10} (1.6) & \textbf{10/10} (3.2) & \textbf{10/10} (5.1) & \textbf{10/10} (2.1) & \textbf{10/10} (5.4) & \textbf{10/10} (9.7) \\
9-G & 5/10 (3.3) & 11/20 (2.8) & 6/10 (2.6) & \textbf{10/10} (\underline{1.2}) & \textbf{10/10} (1.6) & \textbf{10/10} (2.2) & \textbf{10/10} (1.7) & \textbf{10/10} (3.8) & \textbf{10/10} (6.2) & \textbf{10/10} (2.4) & \textbf{10/10} (6.6) & \textbf{10/10} (12) \\
16-G & 5/70 (29.7) & 6/30 (10.2) & 5/50 (24.3) & 8/10 (\underline{1.8}) & \textbf{10/10} (\underline{1.8}) & \textbf{10/10} (2.4) & \textbf{10/10} (\underline{1.9}) & \textbf{10/10} (4.5) & \textbf{10/10} (7.5) & \textbf{10/10} (2.7) & \textbf{10/10} (8) & \textbf{10/10} (14.9) \\
32-G & 1/100 (316.6) & 2/100 (153.4) & 5/100 (61.2) & 7/30 (9.8) & 8/20 (\underline{7.1}) & 11/20 (\underline{6.8}) & 5/20 (14.4) & 7/10 (9.3) & \textbf{10/10} (10.8) & 6/50 (39.4) & \textbf{9/10} (12.8) & \textbf{10/10} (21.7) \\
20-D & 7/20 (4.9) & 5/30 (11.5) & 8/30 (6.4) & \textbf{10/10} (\underline{1.3}) & \textbf{9/10} (2.1) & 8/10 (3.4) & \textbf{10/10} (2) & 8/10 (6.4) & \textbf{9/10} (9.7) & \textbf{10/10} (3) & \textbf{9/10} (10.3) & 5/10 (34.7) \\
24-C & 5/70 (28.1) & 7/20 (3.9) & 5/50 (15.1) & 6/10 (\underline{2.9}) & 8/10 (\underline{2.9}) & 8/10 (4.6) & 9/20 (6.9) & \textbf{10/10} (7.6) & \textbf{9/10} (14.7) & 7/20 (13.5) & \textbf{10/10} (13.8) & \textbf{9/10} (29.8) \\
UDISJ4 & \textbf{9/10} (\underline{1.3}) & \textbf{9/10} (\underline{1.3}) & \textbf{9/10} (\underline{1.4}) & \textbf{10/10} (\underline{1.3}) & \textbf{10/10} (1.9) & \textbf{9/10} (3) & \textbf{10/10} (2) & \textbf{10/10} (5.1) & 8/10 (10.9) & \textbf{10/10} (3) & \textbf{10/10} (9) & 7/40 (98.6) \\
UDISJ5 & 8/10 (\underline{1.8}) & \textbf{9/10} (\underline{1.7}) & 6/10 (3.2) & \textbf{10/10} (2.1) & \textbf{10/10} (4.8) & \textbf{10/10} (8.4) & 8/10 (7) & \textbf{10/10} (19.9) & \textbf{10/10} (35.8) & 5/10 (20.4) & \textbf{10/10} (38.2) & \textbf{10/10} (74.8) \\
UDISJ6 & \textbf{7/10} (\underline{2.6}) & \textbf{7/10} (4) & \textbf{13/20} (6.9) & 7/20 (15.1) & 6/40 (116.8) & 7/30 (149.7) & 5/40 (156.2) & 0/100 ($\sim$) & 0/100 ($\sim$) & 5/40 (293.2) & 0/100 ($\sim$) & 0/100 ($\sim$) \\
RND1 & \textbf{10/10} (\underline{1.1}) & \textbf{10/10} (\underline{1.2}) & \textbf{10/10} (1.3) & \textbf{10/10} (1.3) & \textbf{10/10} (2.2) & \textbf{10/10} (3.1) & \textbf{10/10} (2.4) & \textbf{10/10} (6.4) & \textbf{10/10} (12.2) & \textbf{10/10} (3.7) & \textbf{10/10} (12.4) & \textbf{10/10} (22.4) \\
RND3 & \textbf{10/10} (\underline{1.1}) & \textbf{10/10} (\underline{1.2}) & \textbf{10/10} (1.3) & \textbf{10/10} (1.3) & \textbf{10/10} (2.2) & \textbf{10/10} (3.1) & \textbf{10/10} (2.4) & \textbf{10/10} (6.4) & \textbf{10/10} (12.3) & \textbf{10/10} (3.7) & \textbf{10/10} (12.4) & \textbf{10/10} (22.3) \\
\hline
\end{tabular}
\label{rbrParamKN}
\end{center}
\end{sidewaystable} 

\section{Initialization for the Hybridization}  \label{appD} 

Again the best initialization strategy is SPARSE10. However, it is interesting to note that Hybrid is less sensitive to initialization than SA and RBR. 
In fact, except for 32-G with RNDCUBE and LEDM32 with SPARSE01, it was able to compute exact NMF's in all situations.
In other words, Hybrid is a more robust strategy than RBR and SA although it is computationally more expensive on average. 
\begin{table}[h]
\footnotesize
\begin{center}
\caption{Comparison of the different initialization strategies combined with the hybridization between RBR and SA.}
\begin{tabular}{|c||c|c|c|c|c|}
\hline
 & sparse 00 & sparse 10 & sparse 01 & sparse 11 & rndcube \\
\hline
LEDM6 & \textbf{10/10} (\underline{20.1}) & \textbf{10/10} (\underline{20.4}) & \textbf{10/10} (\underline{20.2}) & \textbf{10/10} (\underline{20}) & \textbf{10/10} (\underline{20.1}) \\
LEDM8 & \textbf{10/10} (59.7) & \textbf{10/10} (59.2) & \textbf{10/10} (\underline{53}) & \textbf{10/10} (65.9) & \textbf{10/10} (61.6) \\
LEDM12 & 7/10 (36) & 5/10 (50.8) & 5/10 (51) & 7/10 (36.7) & \textbf{8/10} (\underline{31.2}) \\
LEDM16 & 5/10 (102.6) & 8/20 (103.1) & 11/20 (91.3) & 5/10 (\underline{69.8}) & \textbf{7/10} (\underline{74.5}) \\
LEDM32 & \textbf{4/100} (\underline{946.2}) & \textbf{2/100} (1851.1) & \textbf{1/100} (3796.3) & 0/100 ($\sim$) & 0/100 ($\sim$) \\
6-G & \textbf{10/10} (\underline{1.5}) & \textbf{10/10} (\underline{1.4}) & \textbf{10/10} (\underline{1.5}) & \textbf{10/10} (\underline{1.5}) & \textbf{10/10} (\underline{1.4}) \\
7-G & \textbf{10/10} (4.5) & \textbf{10/10} (3.1) & \textbf{10/10} (\underline{2.5}) & \textbf{10/10} (3.5) & \textbf{10/10} (3.1) \\
8-G & \textbf{10/10} (\underline{14.7}) & \textbf{10/10} (\underline{13.4}) & \textbf{10/10} (19.4) & \textbf{10/10} (20.2) & \textbf{10/10} (19.2) \\
9-G & \textbf{10/10} (\underline{22.9}) & \textbf{10/10} (\underline{22}) & \textbf{10/10} (\underline{23.8}) & \textbf{10/10} (\underline{24.1}) & \textbf{10/10} (\underline{23.9}) \\
16-G & \textbf{10/10} (\underline{26.4}) & 7/10 (36.7) & 8/10 (34.6) & 6/10 (45.6) & 5/20 (109.3) \\
32-G & \textbf{5/10} (\underline{67.4}) & \textbf{5/10} (\underline{66.9}) & 6/30 (176.6) & 6/40 (235.9) & 0/100 ($\sim$) \\
20-D & \textbf{10/10} (10.1) & \textbf{10/10} (\underline{4.4}) & \textbf{10/10} (10.3) & \textbf{10/10} (6.7) & \textbf{10/10} (5.7) \\
24-C & \textbf{10/10} (\underline{2.7}) & \textbf{10/10} (5.6) & \textbf{10/10} (3.1) & \textbf{10/10} (4.1) & \textbf{10/10} (\underline{2.9}) \\
UDISJ4 & \textbf{10/10} (\underline{1.9}) & \textbf{10/10} (\underline{1.9}) & \textbf{10/10} (\underline{1.9}) & \textbf{10/10} (\underline{1.9}) & \textbf{10/10} (\underline{1.9}) \\
UDISJ5 & \textbf{10/10} (\underline{5.1}) & \textbf{10/10} (\underline{5}) & \textbf{10/10} (\underline{5.3}) & \textbf{10/10} (\underline{5.3}) & \textbf{10/10} (5.8) \\
UDISJ6 & \textbf{10/10} (\underline{18.6}) & \textbf{10/10} (21.2) & \textbf{10/10} (19.3) & \textbf{10/10} (\underline{17.2}) & \textbf{10/10} (21.4) \\
RND1 & \textbf{10/10} (\underline{2.2}) & \textbf{10/10} (\underline{2.2}) & \textbf{10/10} (\underline{2.3}) & \textbf{10/10} (\underline{2.2}) & \textbf{10/10} (\underline{2.2}) \\
RND3 & \textbf{10/10} (\underline{2.2}) & \textbf{10/10} (\underline{2.2}) & \textbf{10/10} (\underline{2.2}) & \textbf{10/10} (\underline{2.2}) & \textbf{10/10} (\underline{2.2}) \\
\hline
\end{tabular}
\label{tableinithyb}
\end{center}
\end{table}

\end{document}